\documentclass[12pt]{amsart}
\usepackage{amsmath,amssymb,mathrsfs,hyperref,color}

\usepackage{paralist}

\usepackage{subfigure}
\usepackage{float}
\usepackage{times}
\usepackage{tikz}
\usepackage{cases}
\usepackage[latin1]{inputenc}
\usetikzlibrary{intersections}
\usepackage{xcolor}
\usepackage{bbm}
\usetikzlibrary{shadings}
\usetikzlibrary[intersections]
\usetikzlibrary{arrows,shapes}

\makeatletter
\def\setliststart#1{\setcounter{\@listctr}{#1}%
  \addtocounter{\@listctr}{-1}}
\makeatother

\makeatletter
\@addtoreset{figure}{section}
\makeatother

\usepackage{calc}
\newtheorem{theorem}{Theorem}[section]
\newtheorem{lemma}[theorem]{Lemma}
\newtheorem{proposition}[theorem]{Proposition}
\newtheorem{corollary}[theorem]{Corollary}
\newtheorem{remark}[theorem]{Remark}
\newtheoremstyle{break}
  {\topsep}{\topsep}%
  {\itshape}{}%
  {\bfseries}{}%
  {\newline}{}%
\theoremstyle{break}

\newtheorem{definition}[theorem]{Definition}
\numberwithin{equation}{section}

\newcommand{\R}{\mathbb{R}}

\newcommand{\N}{\mathbb{N}}

\newcommand{\PP}{\mathcal{P}}

\DeclareMathOperator*{\supp}{supp}
\DeclareMathOperator*{\esssup}{ess\ sup}
\DeclareMathOperator*{\ddiv}{div}

\makeatletter
\def\moverlay{\mathpalette\mov@rlay}
\def\mov@rlay#1#2{\leavevmode\vtop{%
   \baselineskip\z@skip \lineskiplimit-\maxdimen
   \ialign{\hfil$\m@th#1##$\hfil\cr#2\crcr}}}
\newcommand{\charfusion}[3][\mathord]{
    #1{\ifx#1\mathop\vphantom{#2}\fi
        \mathpalette\mov@rlay{#2\cr#3}
      }
    \ifx#1\mathop\expandafter\displaylimits\fi}
\makeatother

\title[Mild and weak solutions of Mean Field Games problem for linear control systems]{Mild and weak solutions of Mean Field Games problem for linear control systems}
\author{Piermarco Cannarsa \and Cristian Mendico}
\address{Dipartimento di Matematica, Universit\`a di Roma ``Tor Vergata'', Via della Ricerca Scientifica 1, 00133 Roma, Italy}
\email{cannarsa@mat.uniroma2.it}
\address{GSSI-Gran Sasso Science Institute, Viale F. Crispi 7, 67100  L'Aquila and CEREMADE, Universit\'e Paris-Dauphine, Place du Maréchal de Lattre de Tassigny - 75775 PARIS Cedex 16 }
\email{cristian.mendico@gssi.it, cristian.mendico@dauphine.eu}

\date{\today}

\begin{document}

\maketitle

\maketitle
	\begin{abstract}
		The aim of this paper is to study first order Mean field games subject to a linear controlled dynamics on $\R^{d}$. For this kind of problems, we define Nash equilibria (called Mean Field Games equilibria), as Borel probability measures on the space of admissible trajectories, and mild solutions as solutions associated with such equilibria. Moreover, we prove the existence and uniqueness of mild solutions and we study their regularity:
		 we prove H\"older regularity of Mean Field Games equilibria and fractional semiconcavity for the value function of the underlying optimal control problem. Finally, we address the PDEs system associated with the Mean Field Games problem and we prove that the class of mild solutions coincides with a suitable class of weak solutions. 
		\end{abstract}

	\tableofcontents
	\section{Introduction}
	
    The goal of this paper is to analyze the first order Mean Field Games problem, using the Lagrangian formalism, where agents are subject to a linear controlled dynamics on $\R^{d}$.

The model we have in mind is a game or a system for which one is interested in controlling not the velocity of each agent but its acceleration. Therefore, the Lagrangian function and terminal cost of the common optimization problem depend on higher order derivates of the admissible paths on $\R^{d}$.

    We recall that Mean Field Games theory has been introduced simultaneously, but independently, by Lasry and Lions in \cite{bib:LL1}, \cite{bib:LL2} and \cite{bib:LL3}, and by M. Huang, R. P. Malham\'e and P. E. Caines in \cite{bib:HCM2} and \cite{bib:HCM1}. This theory is devoted to the study of deterministic and stochastic differential games with a large number of players, where each agent is rational and has a small influence on the whole evolution of the model.

  Fixed a time horizon $T>0$, we consider players subject to the following dynamics
    \begin{equation}\label{introdyn}
    \dot\gamma(t)=A\gamma(t)+B u(t), \quad \forall\ t \in [0,T]
    \end{equation}
where $A$ and $B$ are real matrices and $u$ is an admissible control function.
Each player aims to minimize a cost functional of the form
\begin{equation*}
\int_{0}^{T}{L(\gamma(s), u(s), m_{s})\ ds}+G(\gamma(T), m_{T}),
\end{equation*}
where, for each time $t \in [0,T] $, we have that $m_{t}$ is a flow of Borel probability measures on $\R^{d}$ depending continuously on $t$. More precisely, we define the metric space
\begin{align*}
\Gamma_{T}=\Big\{\gamma \in \text{AC}([0,T]): \gamma(t)\ \text{solution of}\ \eqref{introdyn}  \Big\},
\end{align*}
endowed with the uniform metric $\| \cdot \|_{\infty}$ and we consider Borel probability measures $\eta$ on $\Gamma_{T}$ with a finite first order moment. Then, denoting by  $e_{t}: \Gamma_{T} \to \R^{d}$ evaluation map, we define $m_{t}= e_{t} \sharp \eta$, where $\sharp$ stands for the push-forward operator.

The first problem we deal with, is the definition of Nash equilibria (Mean Field Games equilibria) for this class of problems. Inspired by recent works on Mean Field Games, see for instance \cite{bib:CC} and \cite{bib:GF}, given an initial distribution $m_{0} \in \PP(\R^{d})$ we define Nash equilibria as probability measures supported on minimizing curves of the above functional such that $e_{0} \sharp \eta = m_{0}$.  Then, by a fixed point argument we are able to prove that such measures exist and we find conditions yielding uniqueness, see, respectively, Theorem \ref{existenceofmfg} and Theorem \ref{uniquenessmfg}.

The idea of constructing Nash equilibria by considering measures on path space, which is typical of the Lagrangian approach, can also be found in the so-called probabilistic approach, see for instance \cite{bib:DV}, \cite{bib:DL}, \cite{bib:MF} and \cite{bib:RD}. In such settings, one studies more general stochastic Mean Field Games problems obtaining Nash equilibria as fixed points of certain relaxed functionals on the space of optimally controlled state processes and optimal control processes.

Then, we study the regularity of the so-called mild solutions of the Mean Field Games given by a pair $(V, m) \in C([0,T] \times \R^{d}) \times C([0,T]; \PP_{\alpha}(\R^{d}))$ where $m_{t}$ is the distribution of the agents at time $t \in [0,T]$ and $V$ is the value function of the above optimal control problem.  More precisely, we first prove that the map $t \to m_{t}$ is $\frac{1}{2}$-H\"older continuous in time, see Theorem \ref{holder}. Then, we prove the first main results of the paper which states that the value function $V$ is locally semiconcave on $[0,T] \times \R^{d}$, linearly in space and with a fractional semiconcavity modulus in time, see Theorem \ref{Vsemiconcavity}. 
Moreover, by standard tools of optimal control theory we get that $V$ is locally Lipschitz continuous, see Theorem \ref{Vlipschitz}, on $[0,T] \times \R^{d}$.

 Furthermore, we show that, under some extra assumptions on the Lagrangian function, it is possible to prove that there exists at least one Mean Field Games equilibrium $\eta$ such that the associated evolutionary distribution $m_{t}=e_{t} \sharp \eta$ is Lipschitz continuous in time. Consequently, the 
 first component $V$ of the corresponding mild solution is locally semiconcave in $[0,T] \times \R^{d}$ with a linear modulus of semiconcavity.
In conclusion, we prove the third main result of this paper that is the equivalence between mild solutions and weak solutions of the Mean Field Games system, Theorem \ref{weaksolutions}.

After this paper was submitted and posted on arXiv, similar results were presented in \cite{bib:ACP} for the special case of Mean Field Games with control on acceleration.

	\medskip\medskip\medskip
	The paper is organized as follows: in Section 2, we fix the notation used throughout the paper and we recall some notions and results from measure theory and control theory; in Section 3, we explain the general setting of the problem and we prove some preliminary results which are used later; in Section 4, we prove the existence and uniqueness of Mean Field Games equilibria; in Section 5, we study the regularity of the mild solutions of the Mean Field Games problem; in Section 6, we address the Mean Field Games system derived from the optimal control problem and we analyze the structure of the weak solutions of this system; in the Appendix, we give the proofs of a Lipschitz regularity results that is used in this paper. 
	
	\section{Preliminaries}
	

	\subsection{Notation}
We write below a list of symbols used throughout this paper.
\begin{itemize}
	\item Denote by $\mathbb{N}$ the set of positive integers, by $\mathbb{R}^d$ the $d$-dimensional real Euclidean space,  by $\langle\cdot,\cdot\rangle$ the Euclidean scalar product, by $|\cdot|$ the usual norm in $\mathbb{R}^d$, and by $B_{R}$ the open ball with center $0$ and radius $R$.
	\item Let $\Lambda$ be a real $n\times n$ matrix. Define the norm of $\Lambda$ by 
\[
\|\Lambda\|=\sup_{|x|=1,\ x\in\mathbb{R}^d}|\Lambda x|.
\]

\item Let $A$ be a Lebesgue-measurable subset of $\mathbb{R}^d$. Denote by $\mathcal{L}^{n}(A)$ the $n$-dimensional Lebesgue measure of $A$. Denote by $\mathbf{1}_{A}:\mathbb{R}^n\rightarrow \{0,1\}$ the characteristic function of $A$, i.e.,
\begin{align*}
\mathbf{1}_{A}(x)=
\begin{cases}
1  \ \ \ &x\in A,\\
0 &x \not\in A.
\end{cases}
\end{align*} 

\item Let $f$ be a real-valued function on $\mathbb{R}^d$. The set
\[
D^{*} f(x)=\left\{p\in\mathbb{R}^d: \exists \{x_{k}\}_{k \in \N},\ x_{k} \to x, \ \forall\ k \in \N \ \exists Df(x_{k}),\ Df(x_{k}) \to p  \right\},
\]
is called the set of reachable gradients of $f$ at $x$.

\item Let $A$ be a Lebesgue-measurable subset of $\mathbb{R}^{d}$. Let $1\leq p\leq \infty$. 
Denote by $L^p(A)$ the space of Lebesgue-measurable functions $f$ with $\|f\|_{p,A}<\infty$, where   
\begin{align*}
& \|f\|_{\infty, A}:=\esssup_{x \in A} |f(x)|,
\\& \|f\|_{p,A}:=\left(\int_{A}|f|^{p}\ dx\right)^{\frac{1}{p}}, \quad 1\leq p<\infty.
\end{align*}
Denote $\|f\|_{\infty,\mathbb{R}^d}$ by $\|f\|_{\infty}$ and $\|f\|_{p,\mathbb{R}^d}$ by $\|f\|_{p}$, for brevity.

\item $C_b(\mathbb{R}^d)$ stands for the function space of bounded uniformly  continuous functions on $\mathbb{R}^d$. $C^{2}_{b}(\mathbb{R}^{d})$ stands for the space of bounded functions on $\mathbb{R}^d$ with bounded uniformly continuous first and second derivatives. 
$C^k(\mathbb{R}^{d})$ ($k\in\mathbb{N}$) stands for the function space of $k$-times continuously differentiable functions on $\mathbb{R}^d$, and $C^\infty(\mathbb{R}^{d}):=\cap_{k=0}^\infty C^k(\mathbb{R}^{d})$. 
 $C_c^\infty(\mathbb{R}^{d})$ stands for the space of functions in $C^\infty(\mathbb{R}^{d})$ with compact support. Let $a<b\in\mathbb{R}$.
  $AC([a,b];\mathbb{R}^d)$ denotes the space of absolutely continuous maps $[a,b]\to \mathbb{R}^d$.
  
  \item For $f \in C^{1}(\mathbb{R}^{d})$, the gradient of $f$ is denoted by $Df=(D_{x_{1}}f, ..., D_{x_{n}}f)$, where $D_{x_{i}}f=\frac{\partial f}{\partial x_{i}}$, $i=1,2,\cdots,d$.
Let $k$ be a nonnegative integer and let $\alpha=(\alpha_1,\cdots,\alpha_d)$ be a multiindex of order $k$, i.e., $k=|\alpha|=\alpha_1+\cdots +\alpha_d$ , where each component $\alpha_i$ is a nonnegative integer.   For $f \in C^{k}(\mathbb{R}^{d})$,
define $D^{\alpha}f:= D_{x_{1}}^{\alpha_{1}} \cdot\cdot\cdot D^{\alpha_{d}}_{x_{d}}f$. 
\end{itemize}


\subsection{Measure Theory}

Denote by $\mathscr{B}(\mathbb{R}^d)$ the  Borel $\sigma$-algebra on $\mathbb{R}^d$ and by $\mathcal{P}(\mathbb{R}^d)$ the space of Borel probability measures on $\mathbb{R}^d$.
The support of a measure $\mu \in \mathcal{P}(\mathbb{R}^n)$, denoted by $\supp(\mu)$, is the closed set defined by
\begin{equation*}
\supp (\mu) := \Big \{x \in \mathbb{R}^d: \mu(V_x)>0\ \text{for each open neighborhood $V_x$ of $x$}\Big\}.
\end{equation*}
We say that a sequence $\{\mu_k\}_{k\in\mathbb{N}}\subset \mathcal{P}(\mathbb{R}^d)$ is weakly-$*$ convergent to $\mu \in \mathcal{P}(\mathbb{R}^d)$, denoted by
$\mu_k \stackrel{w^*}{\longrightarrow}\mu$, 
  if
\begin{equation*}
\lim_{n\rightarrow \infty} \int_{\mathbb{R}^d} f(x)\,d\mu_n(x)=\int_{\mathbb{R}^d} f(x) \,d\mu(x), \quad  \forall f \in C_b(\mathbb{R}^d).
\end{equation*}

For $p\in[1,+\infty)$, the Wasserstein space of order $p$ is defined as
\begin{equation*}
\mathcal{P}_p(\mathbb{R}^d):=\left\{m\in\mathcal{P}(\mathbb{R}^d): \int_{\mathbb{R}^d} |x_0-x|^p\,dm(x) <+\infty\right\},
\end{equation*}
where $x_0 \in \mathbb{R}^d$ is arbitrary. 
Given any two measures $m$ and $m^{\prime}$ in $\mathcal{P}_p(\mathbb{R}^n)$,  define
\[
\Pi(m,m^{\prime}):=\Big\{\lambda\in\mathcal{P}(\mathbb{R}^d\times \mathbb{R}^d): \lambda(A\times \mathbb{R}^d)=m(A),\ \lambda(\mathbb{R}^d\times A)=m'(A),\ \forall A\in \mathscr{B}(\mathbb{R}^d)\Big\}.
\]
The Wasserstein distance of order $p$ between $m$ and $m'$ is defined by
    \begin{equation*}\label{dis1}
          d_p(m,m')=\inf_{\lambda \in\Pi(m,m')}\left(\int_{\mathbb{R}^d\times \mathbb{R}^d}|x-y|^p\,d\lambda(x,y) \right)^{1/p}.
    \end{equation*}
    The distance $d_1$ is also commonly called the Kantorovich-Rubinstein distance and can be characterized by a useful duality formula (see, for instance, \cite{bib:CV})  as follows
\begin{equation*}
d_1(m,m')=\sup\left\{\int_{\mathbb{R}^d} f(x)\,dm(x)-\int_{\mathbb{R}^d} f(x)\,dm'(x) \ |\ f:\mathbb{R}^d\rightarrow\mathbb{R} \ \ \text{is}\ 1\text{-Lipschitz}\right\}, 
\end{equation*}
for all $m$, $m'\in\mathcal{P}_1(\mathbb{R}^d)$.

We now recall that weak-$\ast$ convergence is equivalent to convergence in the metric space $(\mathcal{P}_{p}(\mathbb{R}^d), d_{p})$ (see, for instance,  \cite{bib:CV}) and useful compactness criterion for subsets of $\PP_{p}(\R^{d})$.
\begin{proposition}\label{cm}
Let $\{\mu_k\}_{k\in \mathbb{N}}$ be a sequence of measures in $\mathcal{P}_p(\mathbb{R}^d)$ and let $\mu$ be another element of $\mathcal{P}_p(\mathbb{R}^d)$.
Then
\begin{itemize}
	\item [(i)] if $d_p(\mu_k,\mu)\to 0$, then $\mu_k \stackrel{w^*}{\longrightarrow}\mu$,  as $k\to+\infty$;
	\item [(ii)] if $\supp(\mu_k)$ is contained in  a fixed compact subset of $\mathbb{R}^d$ for all $k\in \mathbb{N}$ and $\mu_k \stackrel{w^*}{\longrightarrow}\mu$,  as $k\to+\infty$, then $d_p(\mu_k,\mu)\to 0$, as $k\to+\infty$.
\end{itemize}
 \end{proposition}

Let $\mathcal{K}$ be a subset of $\mathcal{P}(\mathbb{R}^d)$. We say that the set $\mathcal{K}$ has uniformly integrable $p$-moment with respect some (and thus any) $\bar{x} \in \R^{d}$ if and only if 
\begin{equation*}
\lim_{i \to \infty} \int_{\R^{d} \backslash B_{i}(\bar{x})}{|x-\bar{x}|^{p}\ \mu(dx)}=0, \quad \text{uniformly with respect to}\ \mu \in \mathcal{K}.
\end{equation*}
\begin{remark}\label{criterion}
\em{Notice that, if 
\begin{equation*}
0 < p < p_{1}, \quad \text{and} \quad \sup_{\mu \in \mathcal{K}} \int_{\R^{d}}{|x-\bar{x}|^{p_{1}}\ \mu(dx)} < +\infty,
\end{equation*}
then $\mathcal{K}$ has uniformly integrable $p$-moment.}
\end{remark}
\begin{theorem}[{\bf Compactness and convergence}]\label{compactness}
A set $\mathcal{K} \subset \PP_{p}(\R^{d})$ is relatively compact if and only if it is $p$-uniformly integrable and tight. Moreover, for a given sequence $\{ \mu_i \}_{i \in \N} \subset \PP_{p}(\R^{d})$ we have that 
\begin{equation*}
\lim_{i \to \infty} d_{p}(\mu_{i}, \mu)=0
\end{equation*}
if and only if $\mu_{i}$ narrowly converge to $\mu$ and $\{ \mu_i \}_{i \in \N}$ has uniformly integral $p$-moment.
\end{theorem}

\begin{theorem}\label{compactnesscrit}
Let $r \geq p >0$ and let $\mathcal{K} \subset \PP_{p}(\R^{d})$ be such that
\begin{equation*}
\sup_{\mu \in \mathcal{K}}\int_{\R^{d}}{|x|^{r}\ \mu(dx)} < \infty.	
\end{equation*}
	Then the set $\mathcal{K}$ is tight. If, moreover, $r > p$ then $\mathcal{K}$ is relatively compact for the $d_{p}$ distance.
\end{theorem}

See, for instance, \cite[Theorem 7.1.5]{bib:AGS} and \cite[Lemma 5.7]{bib:CN}.

Let $(X_1,S_1,\mu)$ be a measure space, $(X_2,S_2)$ a measurable space, and $f:X_1\to X_2$ a measurable map. The push-forward of $\mu$ through $f$ is the measure $f \sharp \mu$ on $(X_{2}, S_{2})$ defined by 
\begin{equation*}
f \sharp \mu(B):=\mu\left(f^{-1}(B)\right), \quad \forall B \in S_2.
\end{equation*} 
\noindent
The push-forward has the property that a measurable map $g: X_{2} \to \mathbb{R}$ is integrable with respect to $f \sharp \mu$ if and only if $g \circ f$ is integrable on $X_{1}$ with respect to $\mu$. In this case, we have that 
\begin{equation*}
\int_{X_1}g(f(x))\,d\mu(x)=\int_{X_2}g(y)\,df\sharp\mu(y).
\end{equation*}

We conclude this introductory section recalling the so-called disintegration theorem.

\begin{theorem}[{\bf Disintegration Theorem}]\label{disintegration}
Let $X$ and $Y$ be Radon separable metric spaces, let $\mu$ be a Borel probability measure on $X$ and let $\pi:X \to Y $ be Borel map. Define $\nu=\pi \sharp \mu \in \PP(Y)$.  Then there exists a $\mu$-a.e. uniquely determined Borel measurable family of probability measures $\{ \nu_{y} \}_{y \in Y} \subset \PP(X)$ such that 
\begin{equation*}
\nu_{y} ( X \backslash \pi^{-1}(y))=0, \quad\ \text{for}\ \mu-\text{a.e.}\ y \in Y,
\end{equation*}
and
\begin{equation*}
\int_{X}{f(x) \mu(dx)} = \int_{Y} \left(\int_{\pi^{-1}(y)}{f(x) \nu_{y}(dx)} \right) \nu(dy)
\end{equation*}
for every Borel map $f:X \to [0, + \infty]$.
\end{theorem}

See, for instance, \cite[Theorem 5.3.1]{bib:AGS}.


\subsection{Control Theory}

 \begin{definition}[{\bf Strict Tonelli Lagrangians}]\label{def2}
A $C^{2}$ function $L: \mathbb{R}^{n} \times \mathbb{R}^{n} \to \mathbb{R}$ is called a {\it strict Tonelli Lagrangian} if there exist positive constants $C_{i}$ ($i=1,2,3$) such that, for all $(x,v) \in \mathbb{R}^{n} \times \mathbb{R}^{n}$:
\begin{itemize}
\item[(a)] $\frac{I}{C_{1}} \leq D_{vv}^{2}L(x,v) \leq C_{1} I$, where $I$ is the identity matrix;
	\item[(b)] $\|D^{2}_{vx}L(x,v)\| \leq C_{2}(1+|v|)$;
	\item[(c)] $|L(x,0)|+|D_{x}L(x,0)|+ |D_{v}L(x,0)| \leq C_{3}$.  
	 \end{itemize}
\end{definition}

\medskip

Let $L$ be a strict Tonelli Lagrangian and, let $f: \R^{d} \times \R^{k} \to \R^{d}$ and $g:\R^{d} \to \R$ be real functions such that
\begin{itemize}
\item[({\bf f})] for any $u \in \R^{k}$, the map $x \mapsto f(x,u)$ belongs to $W^{1,\infty}(\R^{d}; \R^{d})$ and the gradient $D_{x} f$ exists and is continuous; in addition, there exists a real positive constant $k$ such that $\| D_{x}f(x_{1},u)-D_{x}f(x_{2},u) \| \leq k |x_{1}-x_{2}|$ for all $x_{1}$, $x_{2} \in \R^{d}$ and $u \in \R^{k}$.
\item[({\bf g})] $g \in C^{1}(\R^{d}; \R)$.
\end{itemize}

Define the following optimal control problem
\begin{align*}
\tag{{\bf OC}}
\begin{cases}
& \text{Minimize}\ J(x,u)=g(\gamma(T))+\int_{0}^{T}{L(t, \gamma(t), u(t))\ dt},
\\
&\text{subject to the controlled dynamics}\ \dot\gamma(t)=f(\gamma(t), u(t)), \quad t \in [0,T],
\\
&\text{with constraints}\ \gamma(0)=x
\\
&u(t) \in  \R^{k}, \quad t \in [0,T].
\end{cases}
\end{align*}

Given the optimal control problem {\bf (OC)}, the value function is defined as follows
\begin{equation*}
V(t,x)=\inf_{\substack{u:[0,T] \to \R^{k} \\ \text{measurable}}}\left\{\int_{t}^{T}{L(t, \gamma(t), u(t))\ dt}+g(\gamma(T)) \right\}
\end{equation*}
for every $(t,x) \in [0,T] \times \R^{d}$.
We recall that the value function $V$ satisfies the dynamic programming principle, i.e. for any $(t,x) \in (0,T) \times \R^d$ and any given $s \in (t,T)$ we have that
\begin{equation}\label{DPP}
V(t,x)=\inf_{u: [t,s] \to \R^{k}} \left\{V(s, \gamma(s))+\int_{t}^{s}{L(\tau, \gamma(\tau), u(\tau))\ d\tau} \right\},	
\end{equation}
where $\gamma$ is a solution of the controlled dynamics associated with $u$.

Define the pseudo-Hamiltonian function and the Hamiltonian function as follows:
\begin{align*}
\mathcal{H}(t,x,u,p)=& -\big\langle p, f(x,u) \big\rangle -L(t,x,u), \quad \forall\ (t,x,u,p) \in [0,T] \times \R^{d} \times \R^{k} \times \R^{d}
\\
H(t,x,p)=& \sup_{ u \in \R^{k}} \mathcal{H}(t,x,u,p), \quad \forall\ (t,x,p) \in [0,T] \times \R^{d} \times \R^{d}.
\end{align*}

\begin{theorem}[{\bf Pontryagin maximum principle}]\label{maxprinciple}
Let $L$ be a strict Tonelli Lagrangian. Assume {\bf (f)} and {\bf (g)}. Given $(t,x) \in [0,T] \times \R^{d}$, let $u^{*}:[t,T] \to \R^{k}$ an optimal control for problem {\bf (OC)} with initial point $(t,x)$ and let $\gamma^{*}$ be the corresponding optimal trajectory. Then, there exists an absolutely continuous arc $p: [0,T] \to \R^{d}$ satisfying the following:
\begin{itemize}
\item[(i)] transversality condition: $-p(T)= \nabla g(\gamma^{*}(T))$;
\item[(ii)] the adjoint equation: $\dot p(t)= D_{x} \mathcal{H}(t, \gamma^{*}(t), u^{*}(t), p(t))$ for almost every $t \in [0,T]$;
\item[(iii)] maximum condition: $\mathcal{H}(t, \gamma^{*}(t), u^{*}(t), p(t))=H(t, \gamma^{*}(t), p(t))$.
\end{itemize}
\end{theorem}
See, for instance, \cite[Theorem 7.4.17]{bib:SC}.

Observe that the adjoint equation ($iii$) could be also written  in the following way
\begin{equation*}
-\dot p(t)= D_{x}f(\gamma^{*}(t), u^{*}(t))^{*} p(t)\footnote{Here $D_{x}f^{*}$ denotes the tranpose of $D_{x}f$.}+D_{x}L(t, \gamma^{*}(t), u^{*}(t)).
\end{equation*}

As usual, one can write the maximum principle in form of Hamiltonian system as follows.

\begin{theorem}\label{Hmaxprinciple}
Let $L$ be a strict Tonelli Lagrangian and assume {\bf (f)} and {\bf (g)}. Let $u^{*}$ be an optimal control of the problem {\bf (OC)} and let $\gamma^{*}$ be the associated minimizing curve. Let $p$ be the dual arc given by Theorem \ref{maxprinciple}. Then, the pair $(\gamma^{*}, p)$ solves the system
\begin{align*}
\begin{cases}
\dot\gamma^{*}(t)=-D_{p}H(t, \gamma^{*}(t), p(t)),
\\
\dot p(t)= D_{x}H(t, \gamma^{*}(t), p(t)).
\end{cases}
\end{align*}
Consequently, we have that $\gamma^{*}$ and $p$ belong to $C^{2}([0,T])$.
\end{theorem}


\section{Setting of the Mean Field Games problem}

\subsection{Assumptions}
Throughout this paper we will assume that the Lagrangian $L:\R^{d} \times \R^{k} \times \PP_{1}(\R^{d}) \to \R$ and the function $G: \R^{d} \times \PP_{1}(\R^{d}) \to \R$ satisfy the following: 
\begin{itemize}
\item[{\bf (L1)}] For any $m \in \PP_{1}(\R^{d})$, the map $(x,u) \mapsto L(x, u, m)$ is of class $C^{2}(\mathbb{R}^{d} \times \mathbb{R}^{k})$ and the map $m \mapsto L(x,u,m),$ from $\PP_{1}(\R^{d})$ to $\R$, is Lipschitz continuous with respect to the $d_{1}$ distance, i.e.
\begin{equation*}
Q_{L}:= \displaystyle{\sup_{\substack{(x,u) \in \R^{d} \times \R^{k} \\ m_{1},\ m_{2} \in \PP_{1}(\R^{d}) \\ m_{1} \not= m_{2}}}} \frac{|L(x,u,m_{1})-L(x,u,m_{2})|}{d_{1}(m_{1}, m_{2})} < +\infty.
\end{equation*}
\item[{\bf (L2)}] The map $(x,m) \mapsto G(x,m)$ is of class $C_{b}(\mathbb{R}^{d} \times \mathcal{P}_{1}(\mathbb{R}^{d}))$ and for every $m \in \PP_{1}(\R^{d})$ the map $x \to G(x,m)$ belongs to $C^{1}_{b}(\R^{d})$.
\item[{\bf (L3)}] 
\begin{itemize}
\item[($i$)] There exist a constant $C_{0}$ such that
\begin{align*}
\frac{\text{Id}}{C_{0}} \leq D_{uu}L(x,u,m) \leq C_{0}\text{Id}, \quad \forall\ (x,u,m) \in \R^{d} \times \R^{k} \times \PP_{1}(\R^{d});
\end{align*}
\item[($ii$)] there exists a constant $C_{1} \geq 0$ such that for any $(x,u,m) \in \R^{d} \times \R^{k} \times \PP_{1}(\R^{d})$, it holds $\| D^{2}_{xu}L(x,u,m) \| \leq C_{1}(1+|u|)$;
\item[($iii$)] there exists a constant $C_{2} \geq 0$ such that for any $(x,u,m) \in \R^{d} \times \R^{k} \times \PP_{1}(\R^{d})$
\begin{align*}
|L(x,0,m)|+|D_{x}L(x,0,m)|+|D_{u}(x,0,m)| \leq C_{2};
\end{align*}
\end{itemize}
\end{itemize}

\begin{remark}\label{assumptions}
\em
Note that, in hypothesis {\bf (L3)}, we are assuming that the Lagrangian $L$ is a strict Tonelli Lagrangian, see Definition \ref{def2}, uniformly with respect the measure variable. Moreover, if $L$ satisfies assumptions {\bf (L3)} $(i)$--$(iii)$, then it is not difficult to check that there exist constants $c_{0}$ and $c_{1}$ such that
\begin{align*}
c_{0}|u|^{2}-c_{1} \leq L(x,u,m) \leq c_{1}+\frac{1}{c_{0}}|u|^{2} \quad \forall\ (x,u,m) \in \R^{d} \times \R^{k} \times \PP_{1}(\R^{d}).
\end{align*}
\end{remark}

\medskip
Fix a time horizon $T > 0$. Let $A$ and $B$ be real matrices, $d \times d$ and $d \times k$, respectively. 

Consider the control system defined by
\begin{equation}\label{dyn}
\dot\gamma(t)=A\gamma(t)+Bu(t), \quad t \in [0,T]
\end{equation}
where $u:[0,T] \to \mathbb{R}^{k}$ is a summable function. For all $x \in \R^{d}$ we denote by $\gamma(\cdot\ ; x,u)$ the solution of the differential equation \eqref{dyn} such that $\gamma(0)=x$ and define the metric space 
\begin{align*}
\Gamma_{T}=\Big\{\gamma(\cdot;x,u): x \in \R^{d},\ u \in L^{1}(0,T; \R^{k})  \Big\} \subset \text{AC}([0,T]; \R^{d})
\end{align*}
endowed with the uniform norm, denoted by $\| \cdot \|_{\infty}$. Moreover, define 
\begin{equation*}
\Gamma_{T}(x)=\Big\{\gamma \in \Gamma_{T}: \gamma(0)=x \Big\}.
\end{equation*}

For any $x \in \mathbb{R}^{d}$, any $u \in L^{1}(0,T)$ and any family of Borel probability measures $\{ m_{t}\}_{t \in [0,T]}$ depending continuously on $t$ define the functional 
\begin{align*}
J(x, u, \{m_{t}\}_{t})=\int_{0}^{T}{L(\gamma(t,x,u), u(t), m_{t})\ dt} + G(\gamma(T,x,u), m_{t}),
\end{align*}
and the associated optimal control problem
\begin{equation}\label{min}
\inf_{u \in L^{2}(0,T;\ \mathbb{R}^{k})} J(x, u, \{m_{t}\}_{t}).
\end{equation}
Notice that the restriction to controls $u \in L^{2}(0,T; \R^{k})$ is due to the structure assumptions we imposed on $L$. 
\begin{proposition}\label{bound}
There exists a real positive constant $K$ such that for any $x \in \R^{d}$, any $\{ m_{t}\}_{t \in [0,T]} \subset \PP(\R^{d})$ and any optimal control $u^{*}$ of \eqref{min}, we have that 
    $$\|u^{*} \|_{2} \leq K.
    $$ 
\end{proposition}
\begin{proof}
By Remark \ref{assumptions} and the optimality of $u^{*}$ we have that 
\begin{equation*}
 c_{1}T + \|G\|_{\infty} \geq J_{\eta}(x, 0) \geq J_{\eta}(x, u^*) \geq c_{0} \int_{0}^{T}{|u^{*}(t)|^{2}dt} - c_{1}T - \|G\|_{\infty}.
\end{equation*}
Therefore, from the above inequalities we deduce that
\begin{equation*}
	\| u^{*} \|^{2}_{2}=\int_{0}^{T}{|u^{*}(t)|^{2}dt} \leq \frac{2}{c_{0}}\left(c_{1}T + \|G\|_{\infty}\right)=:K^{2}.
\end{equation*}
Thus, the proof is complete.
\end{proof}
\begin{corollary}\label{infboundmin}
For any $x \in \R^{d}$, let $u^{*}$ be a solution of \eqref{min} and let $\gamma^{*}(\cdot)=\gamma(\cdot\ ;x,u^{*})$. Then, there exists a constant $\tilde{C} \geq 0$ such that
\begin{equation*}
\| \gamma^{*} \|_{\infty} \leq \tilde{C_{1}}(1+|x|).
\end{equation*}
\end{corollary}
\begin{proof}
Since $\gamma^{*}$ is a solution of \eqref{dyn} associated with $u^{*}$, we know that
\begin{equation*}
\gamma^{*}(t)=e^{tA}x+\int_{0}^{t}{e^{(t-s)A}Bu^{*}(s)\ ds}.
\end{equation*}
Thus, we have that
\begin{equation*}
|\gamma^{*}(t)| \leq e^{T\| A \|}\left( |x| + \| B \| \int_{0}^{t}{|u^{*}(s)|\ ds}\right)
\end{equation*}
and by H\"older's inequality
\begin{equation*}
|\gamma^{*}(t)| \leq e^{T\| A \|}\left( |x| + \| B \| T^{\frac{1}{2}}\| u^{*} \|_{2}\right).
\end{equation*}
Thus, the proof is complete.
\end{proof}
\begin{lemma}\label{equicontinuity}
	Let $u^{*} \in L^{2}$ be an optimal control and $\gamma^{*} \in \Gamma^{*}(x)$ be an optimal path for $x \in \R^{d}$. Then, there exists a constant $\tilde{C}_{2}>0$ such that
	\begin{equation*}
		\| \dot\gamma^{*}\|_{2} \leq \tilde{C}_{2}(1+|x|).
	\end{equation*}
Moreover, the family of minimizing path $\Gamma^{*}(x)$ is uniformly H\"older continuous.
\end{lemma}
\begin{proof}
First, we have that by Proposition \ref{bound} and Corollary \ref{infboundmin} the following estimates holds true 
\begin{align*}
& \| \dot\gamma^{*}\|_{2} = \| A\gamma^{*}(t)+Bu^{*}(t) \|_{2} \leq \| A\|^{\frac{1}{2}} \| \gamma^{*} \|_{2} + \| B\|^{\frac{1}{2}} \| u^{*}\|_{2}
\\
\leq & \| A\|^{\frac{1}{2}} \left(\int_{0}^{T}{|\gamma^{*}(t)|^{2}dt}\right)^{\frac{1}{2}} + \| B\|^{\frac{1}{2}} K
\\
\leq & \| A\|^{\frac{1}{2}} T^{\frac{1}{2}}\tilde{C}\big(1+|x|\big) + \| B\|^{\frac{1}{2}} K.
\end{align*}
Thus, for any $t$, $s \in [0,T]$ such that $s \leq t$ we get
\begin{align*}
	&  |\gamma^{*}(t)-\gamma^{*}(s)| \leq \int_{s}^{t}{|\dot\gamma^{*}(\tau)|\ d\tau}
	 \\
	 \leq\ & \| \dot\gamma^{*}\|_{2}|t-s|^{\frac{1}{2}} \leq  \Big(\| A\|^{\frac{1}{2}} T^{\frac{1}{2}}\tilde{C}\big(1+|x|\big) + \| B\|^{\frac{1}{2}} K \Big)|t-s|^{\frac{1}{2}}.
\end{align*}
\end{proof}
In order to construct the Lagrangian formulation of our Mean Field Games problem we are going to give a special structure to the family of probability measures $\{m_{t}\}_{t \in [0,T]}$.
Let $\alpha > 1$ and let $m_{0}$ be a Borel probability measure in $\mathcal{P}_{\alpha}(\mathbb{R}^{d})$. Denote by $[m_{0}]_{\alpha}$ the $\alpha$-moment of $m_{0}$, i.e. 
\begin{equation}\label{mom}
[m_{0}]_{\alpha}=\int_{\mathbb{R}^{d}}{|x|^{\alpha}\ m_{0}(dx)}.
\end{equation} 
Let $R$ be a real constant such that $R \geq [m_{0}]_{\alpha}$ and define the following space of probability measures on $\Gamma_{T}$
\begin{align*}
\mathcal{P}_{m_{0}}(\Gamma_{T}, R)=\left\{ \eta \in \mathcal{P}(\Gamma_{T}) :\int_{\Gamma_{T}}{\| \dot\gamma \|_{2}^{\alpha}\ \eta(d\gamma)}\leq R, \ e_{0} \sharp \eta=m_{0} \right\}
\end{align*}
where $e_{t}(\gamma)=\gamma(t)$ is the evaluation map. Note that the sets $\PP_{m_{0}}(\Gamma_{T}, R)$ are compact subsets of $\PP(\Gamma_{T})$. Indeed, define $\mathcal{C}_{r}$ for any $r>0$ the following sets
\begin{equation*}
\mathcal{C}_{r}=\{\gamma \in \Gamma_{T}: |\gamma(0)| \leq r,\ \|\dot\gamma\|_{2} \leq r\}.	
\end{equation*}
  which are compact by Ascoli-Arzela Theorem. Moreover, observe that by definition
  \begin{equation*}
  \mathcal{C}_{r}^{c} \subset \{\gamma \in \Gamma_{T}: \| \dot\gamma\|_{2} > r\} \cup \{\gamma \in \Gamma_{T}: |\gamma(0)| > r \}.	
  \end{equation*}
   Then, given $\eta \in \PP_{m_{0}}(\Gamma_{T},R)$ we have that by definition 
   \begin{equation*}
   \eta(\{ \gamma \in \Gamma_{T}: |\gamma(0)|>r\}) =m_{0}(B^{c}_{r})	
   \end{equation*}
   which goes to zero as $r \to +\infty$ and by Bienaym\'e-Tchebychev  inequality we have that
  \begin{equation*}
  	\eta(\{ \gamma \in \Gamma_{T}: \|\dot\gamma\|_{2} > r\}) \leq \frac{R}{r^{\alpha}}
  \end{equation*}
  Thus, we get
  \begin{equation*}
  \eta(\mathcal{C}_{r}^{c}) \leq \frac{R}{r^{\alpha}}+m_{0}(B_{r}^{c}).	
  \end{equation*}
Therefore, we deduce that $\PP_{m_{0}}(\Gamma_{T}, R)$ is compact since it is tight. 
 
\begin{remark}\label{empty}
	\em{There exist at least one constant $R \geq [m_{0}]_{\alpha}$ such that the set $\mathcal{P}_{m_0} (\Gamma_{T}, R)$ is non-empty. Indeed, fixed a Borel probability measure $m_{0} \in \PP_{\alpha}(\R^{d})$, consider the map $p: \mathbb{R}^{d} \to \Gamma_{T}$ such that $$x \mapsto p[x](t):= e^{tA}x,\quad \forall t \in [0,T]$$  and define the measure $\eta=p \sharp m_{0} \in \mathcal{P}(\Gamma_{T})$. Note that, for any $x \in \R^{d}$ the curve $e^{tA}x$ is an admissible curve associated with the control $u \equiv 0$.

Then, the following holds:
	\begin{enumerate}
	\item for any bounded continuous function $f$ on $\mathbb{R}^{d}$, we have that $e_{0} \sharp \eta = m_{0}$. Indeed,
	\begin{align*}
	& \int_{\mathbb{R}^{d}}{f(x)\ e_{0} \sharp \eta(dx)}= \int_{\Gamma_{T}}{f(\gamma(0))\ \eta(d\gamma)} 
	\\
	=& \int_{\Gamma_{T}}{f(\gamma(0))\ p\sharp m_{0}(d\gamma)} = \int_{\mathbb{R}^{d}}{f(p[x](0))\ m_{0}(dx)}
	\\
	=& \int_{\mathbb{R}^{d}}{f(x)\ m_{0}(dx)};
	\end{align*}
	\item the $\alpha$-moment of $\eta$ is bounded:
	\begin{align*}
		& \int_{\Gamma_{T}}{\|\dot\gamma \|_{2}^{\alpha}\ \eta(d\gamma)} = \int_{\mathbb{R}^{d}}{\|\dot p[x] \|_{2}^{\alpha}\ m_{0}(dx)}
		\\
		 \leq\ & \left( \|A\|e^{T \|A\|}\right)^{\alpha}\int_{\mathbb{R}^{d}}{|x|^{\alpha}\ m_{0}(dx)} \leq \left( \|A\|e^{T \|A\|}\right)^{\alpha}[m_{0}]_{\alpha}.
	\end{align*}
	\end{enumerate}
Therefore, taking $R \geq \left( \|A\|e^{T \|A\|}\right)^{\alpha}[m_{0}]_{\alpha}$ we have that $\eta \in \mathcal{P}_{m_{0}}(\Gamma_{T}, R)$.} $\quad\quad\quad\quad\quad\quad\quad\quad\quad\quad\quad\square$
\end{remark}

\subsection{Definitions and first properties}
For any $x \in \mathbb{R}^{d}$, any $u \in L^{1}(0,T)$ and any $\eta \in \mathcal{P}_{m_{0}}(\Gamma_{T}, R)$, define the functional 
\begin{align*}
J_{\eta}(x, u)=\int_{0}^{T}{L(\gamma(t,x,u), u(t), e_{t} \sharp \eta)\ dt} + G(\gamma(T,x,u), e_{T} \sharp \eta)
\end{align*}
and the associated optimal control problem
\begin{equation}\label{min}
\inf_{u \in L^{2}(0,T;\ \mathbb{R}^{k})} J_{\eta}(x, u).
\end{equation}
Notice that the restriction to controls $u \in L^{2}(0,T; \R^{k})$ is due to the structure assumptions we imposed on $L$.

We denote by $\Gamma_{\eta}^{*}(x)$ the set of curves associated with an optimal control $u^{*}$, i.e.
\begin{equation*}
\Gamma^{*}_{\eta}(x)=\Big\{\gamma(\cdot\ ;x,u^{*}): J_{\eta}(x,u^{*})=\inf_{u \in L^{2}(0,T; \R^{k})} J_{\eta}(x,u) \Big\}.
\end{equation*}

\begin{definition}[{\bf Mean Field Games equilibrium}]
Given $m_{0} \in \mathcal{P}_{\alpha}(\mathbb{R}^{d})$, we say that $\eta \in \mathcal{P}_{m_{0}}(\Gamma_{T}, R)$ is a Mean Field Games equilibrium for $m_{0}$ if
\begin{equation*}
\supp(\eta) \subset \bigcup_{x \in \mathbb{R}^{d}} \Gamma^{*}_{\eta}(x).
\end{equation*}
\end{definition}


\begin{proposition}\label{measureregularity}
Under the above assumptions the following holds true.
\begin{enumerate}
    \item For any  $\eta \in \mathcal{P}_{m_{0}}(\Gamma_{T}, R)$ we have that 
    \begin{equation}\label{momt}
    \displaystyle{\sup_{t \in [0,T]}} \int_{\mathbb{R}^{d}}{|x|^{\alpha}\ e_{t} \sharp \eta(dx)} \leq R.
    \end{equation}
    Consequently, the family of measures $\{ e_{t} \sharp \eta \}_{t \in [0,T]}$ is compact.
	\item For any $\{ \eta_{i}\}_{i \in \mathbb{N}} \subset \mathcal{P}_{m_{0}}(\Gamma_{T}, R)$ and $\eta \in \mathcal{P}_{m_{0}}(\Gamma_{T}, R)$ such that $\eta_{i} \rightharpoonup^{*} \eta$ we have that $$d_{1}(e_{t} \sharp \eta_{i}, e_{t} \sharp \eta) \to 0$$ for every $t \in [0,T]$.
    \item For any $\eta \in \mathcal{P}_{m_{0}}(\Gamma_{T},R)$ we have that the map $t \in [0,T] \mapsto e_{t} \sharp \eta$ is continuous.
\end{enumerate}
\end{proposition}
\begin{proof}
We are going to prove only the point (1), see \cite[Lemma 3.2]{bib:CC} for a proof of (2) and (3).
\begin{enumerate}
    \item Given $\eta \in \mathcal{P}_{m_{0}}(\Gamma_{T}, R)$ we have that
    \begin{align*}
     \int_{\mathbb{R}^{d}}{|x|^{\alpha}\ e_t \sharp \eta(dx)} = \int_{\Gamma_{T}}{|\gamma(t)|^{\alpha}\ \eta(d\gamma)}\leq \int_{\Gamma_{T}}{\| \gamma \|_{\infty}^{\alpha}\ \eta(d\gamma)} \leq C_{0},
    \end{align*}
    where the last inequality holds by definition of $\PP_{m_{0}}(\Gamma_{T}, R)$.
   So, by Theorem \ref{compactnesscrit} the family of measures $\{ e_{t} \sharp \eta\}_{t \in [0,T]}$ is compact in $\PP_{\alpha}(\R^{d})$ with respect to the $d_{1}$ distance since by assumption $\alpha > 1$.
\end{enumerate}	
\end{proof}

 \begin{remark}
 Note that, in \eqref{momt} the constant $R$ in independent of $t \in [0,T]$ and of $\eta$. Indeed, as explained so far it is fixed a priori such that $R \geq [m_{0}]_{\alpha}$. 
 \end{remark}

\section{Mean Field Games equilibria: Existence and Uniqueness}

At this point, it is not difficult to prove that for any given $\alpha >0 $ and any given initial measure $m_{0} \in \PP_{\alpha}(\R^{d})$ there exists $R_{0} \geq 0$ such that for any $R \geq R_{0}$ there exists at least one Mean Field Games equilibrium $\eta \in \PP_{m_{0}}(\Gamma_{T}, R)$ and that, under a classical monotonicity assumption, such an equilibrium is unique.

For the sake of completeness, we give below the key ideas and steps to prove the existence of a Mean Field Games equilibrium, following the appoach in  \cite{bib:CC}.

Given $m_{0} \in \mathcal{P}_{\alpha}(\mathbb{R}^{d})$ and given $\eta \in \mathcal{P}_{m_{0}}(\Gamma_{T}, R)$ we recall that by the Theorem \ref{disintegration} there exists a unique Borel measurable family of probability measures $\{\eta_{x}\}_{x \in \R^{d}}$ on $\Gamma_{T}$ such that 
\begin{align*}
\eta(d\gamma) =& \int_{\mathbb{R}^{d}}{\eta_{x}(d\gamma)\ m_{0}(dx)}
\\
\supp(\eta_{x}) \subset\ &\ \Gamma_{T}(x), \quad m_{0}-\text{a.e.},\ x \in \mathbb{R}^{d}.
\end{align*}
Define the set-valued map 
\begin{equation*}
E:\ \big(\PP_{m_{0}}(\Gamma_{T}, R), d_{1} \big) \rightrightarrows \big( \PP_{m_{0}}(\Gamma_{T}, R), d_{1} \big)
\end{equation*}
that associates with any $\eta \in \mathcal{P}_{m_{0}}(\Gamma_{T}, R)$ the set
\begin{equation*}
E(\eta)=\Big\{\nu \in \mathcal{P}_{m_{0}}(\Gamma_{T}, R):\  \supp(\nu_{x}) \subset\ \Gamma^{*}_{\eta}(x),\ m_{0}-\text{a.e.}  \Big\}.
\end{equation*}
It is easy to realize that a given $\eta \in \PP_{m_{0}}(\Gamma_{T}, R)$ is a Mean Field Games equilibrium if and only if $\eta$ is a fixed point of the above set-valued map, that is, $\eta \in E(\eta)$. Therefore, in order to prove the existence of Mean Field Games equilibria, we appeal to Kakutani-Fan-Glicksberg's fixed point theorem, see for instance \cite[Corollary 17.55]{bib:AB}, which provides conditions under which the set-valued map $E$ has a fixed point.

We check the validity of such conditions in the following Lemmas.


\begin{lemma}\label{closedlemma}
Let $R \geq [m_{0}]_{\alpha}$.  For any $x_{i} \to x$ in $\mathbb{R}^{d}$, for any $\eta_{i} \rightharpoonup^{*} \eta$ in $\mathcal{P}_{m_{0}}(\Gamma_{T}, R)$ and for any $\gamma_{i} \in \Gamma_{\eta_{i}}^{*}(x_{i})$ such that $\gamma_{i} \to \gamma$ in $\Gamma_{T}$ we have that $\gamma \in \Gamma_{\eta}^{*}(x)$.
\end{lemma}

\begin{proof}
	Since $\gamma_{i} \in \Gamma_{\eta_{i}}^{*}(x_{i})$ we know that there exists a sequence of optimal controls $u_{i} \in L^{2}(0,T)$ such that $\gamma_{i}(\cdot)=\gamma_{i}(\cdot, x_{i}, u_{i})$ for every $t \in [0,T]$. Moreover, from Proposition \ref{bound} we get that $\| u_{i} \|_{2} \leq K$.
	Therefore, up to a subsequence, we obtain that there exists $\bar{u} \in L^{2}(0,T)$ such that $u_{i} \rightharpoonup \bar{u}$ in $L^{2}$. Hence, we are reduced to prove that 
	\begin{enumerate}
		\item $\bar{\gamma}(\cdot)=\gamma(\cdot, x, \bar{u})$;
				\item $J_{\eta}(x,\bar{u}) \leq J_{\eta}(x, u)$ for every $u \in L^{2}([0,T])$,
\end{enumerate}
	\vspace{0.3cm}
	\underline{\bf Point 1:}
	\vspace{0.3cm}
 
	By definition of $\gamma_{i}$, we obtain that
	\begin{equation*}
	\gamma_{i}(t)=e^{At}x+\int_{0}^{t}{e^{A(t-s)}Bu_{i}(s)\ ds}.	
	\end{equation*}
    Let $v$ be a vector on $\R^{d}$, then
    \begin{align*}
    \langle v, \gamma_{i}(t) \rangle =&\ \langle v, e^{At}x \rangle	+ \int_{0}^{t}{\langle v, e^{A(t-s)}Bu_{i}(s) \rangle\ ds}
    \\ 
    =&\ \langle v, e^{At}x \rangle	+ \int_{0}^{t}{\langle (e^{A(t-s)}B)^{*}v, u_{i}(s) \rangle\ ds}.
    \end{align*}
    Thus, letting $i \to \infty$ by the weak $L^{2}$ convergence of $u_{i}$ we obtain that 
    \begin{align*}
    	\langle v, \bar\gamma(t) \rangle = \langle v, e^{At}x \rangle	+ \int_{0}^{t}{\langle f(t), e^{A(t-s)}B\bar{u}(s) \rangle\ ds}.
    \end{align*}
    This concludes the proof of point 1.

	\vspace{0.3cm}
	\noindent
	\underline{\bf Point 2:}
	\vspace{0.3cm}
	
	We now prove  that 
	\begin{equation*}
	J_{\eta}(x, \bar{u}) \leq \liminf_{i \to \infty} J_{\eta_{i}}(x_{i}, u_{i}).	\end{equation*}
       By assumptions on $G$, it follows that
       \begin{equation*}
       G(\gamma_{i}(T), e_{T} \sharp \eta_{i}) \to G(\gamma(T), e_{T} \sharp \eta).	
       \end{equation*}
       Therefore, it suffices to prove that
       \begin{equation*}
      \int_{0}^{T}{L(\bar\gamma(t), \bar{u}(t), e_{t} \sharp \eta)\ dt} \leq \liminf_{i \to \infty} \int_{0}^{T}{L(\gamma_{i}(t), u_{i}(t), e_{t} \sharp \eta_{i})\ dt}.\end{equation*}
Now, 
\begin{align*}
& \int_{0}^{T}{\big(L(\bar\gamma(t), \bar{u}(t), e_{t} \sharp \eta)-L(\gamma_{i}(t), u_{i}(t), e_{t} \sharp \eta) \big)\ dt}
\\
=& \underbrace{\int_{0}^{T}{\big(L(\bar\gamma(t), \bar{u}(t), e_{t} \sharp \eta)-L(\bar\gamma(t), u_{i}(t), e_{t} \sharp \eta) \big)\ dt}}_{\bf A}
\\
+& \underbrace{\int_{0}^{T}{\big(L(\bar\gamma(t), u_{i}(t), e_{t} \sharp \eta)-L(\gamma_{i}(t), u_{i}(t), e_{t} \sharp \eta_{i}) \big)\ dt}}_{\bf B}.
\end{align*}
 By assumption {\bf (L3)} $(iii)$ and Lipschitz condition {\bf (L1)} it follows that $\text{\bf B} \to 0$ as $i \to 0$.   Thus, we have to prove now that the functional 
 \begin{equation*}
\Lambda (u)=\int_{0}^{T}{L(\bar\gamma(t), u(t), e_{t} \sharp \eta)\ dt}
\end{equation*}
is weakly lower semicontinuous with respect to the $L^{2}$ topology. Define, for every $\lambda \in \mathbb{R}$,
	\begin{equation*}
	X_{\lambda}=\{u \in L^{2}(0,T) : \Lambda(u) \leq \lambda \}.	\end{equation*}
By assumption {\bf (L3)} on convexity of the Lagrangian $L$ with respect to controls, we get that the sets $X_{\lambda}$ are convex. Furthermore, such sets are closed in the strong $L^{2}$ topology. Indeed, if $\{ u_{i} \}_{i \in \mathbb{N}} \subset X_{\lambda}$ is such that $u_{i} \to u_{\infty}$ in $L^{2}$ then $u_{i} \to u_{\infty}$ a.e. up to a subsequence. Thus, by the continuity of $L$ we have that $L(\gamma(t), u_{i}(t), e_{t} \sharp \eta) \to L(\gamma(t), u_{\infty}(t), e_{t} \sharp \eta)$ a.e. and by the growth assumption $L$ is bounded from below. Therefore, by Fatou's Lemma we obtain that $u_{\infty} \in X_{\lambda}$. 
Hence, since the sets $X_{\lambda}$ are convex and strongly closed it implies that they are closed also in the $L^{2}$ weak topology. This concludes the proof of point 2.
\end{proof}

\begin{corollary}\label{closed1}
The set-valued map 
\begin{align*}
\phi: \big(\R^{d}, | \cdot | \big) & \rightrightarrows  \big(\Gamma_{T}, \| \cdot \|_{\infty} \big)
\\
x & \mapsto  \Gamma_{\eta}^{*}(x)
\end{align*}
has closed graph.	
\end{corollary}


\begin{lemma}\label{convex}
There exists a constant $R(\alpha, [m_{0}]_{\alpha}) > 0$ such that if $R \geq R(\alpha, [m_{0}]_{\alpha})$ then $E(\eta)$ is non-empty. Moreover, $E(\eta)$ is convex and compact.	
\end{lemma}
\begin{proof}
We, first, prove that given $m_{0} \in \PP_{\alpha}(\R^{d})$ for any $\eta \in \PP_{m_{0}}(\Gamma_{T}, R)$ the set $E(\eta)$ is non empty for some constant $R \geq [m_{0}]_{\alpha}$. 
Indeed, we have that by Corollary \ref{closed1} and \cite[Proposition 9.5]{bib:C} the set-valued map $x \rightrightarrows \Gamma_{\eta}^{*}(x)$ is measurable with closed values. Thus, by \cite[Theorem A 5.2]{bib:SC}, there exists a measurable selection $\tilde\gamma_{x} \in \Gamma_{\eta}^{*}(x)$, that is $\tilde\gamma_{x}(t)=\tilde\gamma(t,x,u^{*})$ for some $u^{*} \in L^{2}(0,T)$ solution of \eqref{min} associated with $\eta$. Define, now, the measure $\tilde\eta$ as follows
\begin{equation*}
\tilde\eta(A)=\int_{\mathbb{R}^{d}}{\delta_{\tilde\gamma_{x}}(A)\ m_{0}(dx)} \quad \text{for any}\ A \in \mathcal{B}(\Gamma_{T}).
\end{equation*}
Thus, we need to prove that $\tilde{\eta} \in \PP_{m_{0}}(\Gamma_{T}, R)$. Indeed, $e_{0} \sharp \tilde{\eta} = m_{0}$ by definition and 
\begin{align*}
\int_{\Gamma_{T}}{\| \dot\gamma \|_{2}^{\alpha}\ \tilde{\eta}(d\gamma)} 
=  \int_{\R^{d}}{ \| \dot{\tilde\gamma}_{x} \|_{2}^{\alpha}\ m_{0}(dx)}  \leq \int_{\R^{d}}{\tilde{C_{2}}^{\alpha}\left(1+|x| \right)^{\alpha}\ m_{0}(dx)},
\end{align*}
where the last inequality holds by Lemma \ref{equicontinuity}. 
Therefore, we deduce that
\begin{equation*}
\int_{\Gamma_{T}}{\| \dot\gamma \|_{2}^{\alpha}\ \tilde{\eta}(d\gamma)} \leq \tilde{C}_{2}^{\alpha} \left(\int_{\R^{d}}{|x|^{\alpha}\ m_{0}(dx)} +1 \right) \leq \tilde{C}_{2}^{\alpha}([m_{0}]_{\alpha}+1).
\end{equation*}
Hence, taking $R \geq R(\alpha, [m_{0}]_{\alpha})$, where $$R(\alpha, [m_{0}]_{\alpha}):=  \tilde{C}_{2}^{\alpha}([m_{0}]_{\alpha}+1)$$ we obtain that $\tilde{\eta} \in \PP_{m_{0}}(\Gamma_{T}, R)$. Consequently, that $E(\eta)$ is non-empty.
The proof of convexity is a straightforward application of \cite[Lemma 3.5]{bib:CC}. In conclusion, for any $\eta \in \PP_{m_{0}}(\Gamma_{T}, R)$ the sets $E(\eta)$ are compact, with respect to the $d_{1}$ distance, since $E(\eta) \subset \PP_{m_{0}}(\Gamma_{T}, R)$ which is compact.
\end{proof}


\begin{lemma}\label{mfgclosedgraph}
For any $R \geq R(\alpha, [m_{0}]_{\alpha})$, the set-valued map 
\begin{align*}
E: \big(\PP_{m_{0}}(\Gamma_{T}, R), d_{1} \big) & \rightrightarrows \big( \PP_{m_{0}}(\Gamma_{T}, R), d_{1} \big)
\\
\eta &\mapsto E(\eta)
\end{align*}
has closed graph.
\end{lemma}
\begin{proof}
The proof of this Lemma is a straightforward application of \cite[Lemma 3.6]{bib:CC}.
\end{proof}


\begin{theorem}[{\bf Existence of Mean Field Games equilibria}]\label{existenceofmfg}
Let $R \geq R(\alpha, [m_{0}]_{\alpha})$, where $R(\alpha, [m_{0}]_{\alpha})$ is defined as in Lemma \ref{convex}. Then, the set-valued map $E$ has a fixed point.	
\end{theorem}
\begin{proof}
By the above lemmas the assumptions of Kakutani's fixed point theorem (see, for instance, \cite{bib:BK}) are satisfied and therefore, there exists a fixed point of the map $E$, that is $\bar{\eta} \in E(\bar{\eta})$ and $\bar{\eta}$ is a Mean Field Games equilibrium.
\end{proof}


At this point, for $\alpha > 1$ fix $m_{0} \in \PP_{\alpha}(\R^{d})$ and $R \geq R(\alpha, [m_{0}]_{\alpha})$, where $R(\alpha, [m_{0}]_{\alpha})$ is defined as in Lemma \ref{convex}. Thus, by Theorem \ref{existenceofmfg} we have that there exists at least one Mean Field Games equilibrium $\eta \in \PP_{m_{0}}(\Gamma_{T}, R)$.

From now on, we denote by $\gamma(s;t,x,u)$ the solution to the following control system
\begin{align}\label{dyn2}
\begin{split}
\begin{cases}
\dot\gamma(s)=A\gamma(s)+Bu(s), \quad s \in [t,T]
\\
\gamma(t)=x,
\end{cases}
\end{split}
\end{align}
where $u: [t,T] \to \R^{k}$ belongs to $L^{2}(t,T; \R^{k})$. Moreover, we introduce the following notation
\begin{equation}\label{notation}
m^{\eta}_{t}=e_{t} \sharp \eta,
\end{equation}
for any $\eta \in \PP_{m_{0}}(\Gamma_{T}, R)$.

\medskip
\begin{definition}[{\bf Mild solutions of Mean Field Games problem}]\label{mildsolution}
We say that $(V,m) \in C([0,T] \times \mathbb{R}^{d}) \times C([0,T], \mathcal{P}_{\alpha}(\mathbb{R}^{d}))$ is a mild solution for the Mean Field Games problem if there exists a Mean Field Games equilibrium $\eta \in \mathcal{P}_{m_{0}}(\Gamma_{T})$ such that
\begin{itemize}
\item[($i$)] $m_{t}=m^{\eta}_{t}$ for all $t \in [0,T]$;
\item[($ii$)] $V$ can be represented as the value function of the optimal control problem \ref{min}, that is
\begin{equation}\label{mfgvaluefunction}
	V(t,x)=\inf_{u \in L^{2}(0,T;\ \mathbb{R}^{k})}\left\{\int_{t}^{T}{L(\gamma(s; t,x,u), u(s), m^{\eta}_{s})\ ds}+G(\gamma(T;t,x,u), m^{\eta}_{T})\right\}
\end{equation}
for all $(t,x) \in [0,T]\times \mathbb{R}^{d}$.
\end{itemize}	
\end{definition}

Note that the above definition is well-posed since we have proved so far that there exists at least one Mean Field Games equilibrium and the map 
\begin{align*}
[0,T] & \to \PP_{\alpha}(\R^{d})
\\
t & \mapsto e_{t} \sharp \eta
\end{align*}
is continuous with respect to $d_{1}$. Moreover, for the same reasons we know that there exists at least one mild solution of the Mean Field Games problem.

In order to study the uniqueness of the Mean Field Games equilibrium, we focus the attention on a particular Lagrangian function, that is
\begin{equation}\label{splittedlagrangian}
L(x,u,m):=\ell (x,u)+F(x, m),
\end{equation}
where $\ell$ and $F$ satisfy the assumptions {\bf (L1)}--{\bf (L3)}.

\begin{definition}[{\bf Monotonicity}]\label{strictmonotonicity}
We say that $\Psi: \mathbb{R}^{d} \times \mathcal{P}(\mathbb{R}^{d}) \to \mathbb{R}$ is monotone if
\begin{equation}\label{mon}
\int_{\mathbb{R}^{d}}{\Big(\Psi(x, m_{1})-\Psi(x,m_{2}) \Big)\ (m_{1}-m_{2})(dx)} \geq 0,	
\end{equation}
for all $m_{1},m_{2} \in \mathcal{P}(\mathbb{R}^{d})$.

We say that $\Psi$ is strictly monotone if \eqref{mon} holds true and 
\begin{equation*}
\int_{\mathbb{R}^{d}}{\Big(\Psi(x, m_{1})-\Psi(x,m_{2}) \Big)\ (m_{1}-m_{2})(dx)} = 0 \iff F(x,m_{1})=F(x,m_{2}),\ \forall\ x \in \R^{d}.
\end{equation*}
\end{definition}

\medskip
\begin{theorem}[{\bf Uniqueness of mild solutions}]\label{uniquenessmfg}
Let $F$ and $G$ be strictly monotone. Then, for any Mean Field Games equilibria $\eta_{1}$ and $\eta_{2}$ in $\PP_{m_{0}}(\Gamma_{T}, R)$ we have that the associated functionals $J_{\eta_{1}}$ and $J_{\eta_{2}}$ are equal.

Consequently, if $(V_{1}, m_{1})$ and $(V_{2}, m_{2})$ are two mild solutions associated with the Mean Field Games equilibria $\eta_{1}$ and $\eta_{2}$, then $V_{1}=V_{2}$.
\end{theorem}
We omit the proof of the Theorem \ref{uniquenessmfg} which is similar to the one of \cite[Theorem 4.1]{bib:CC}.


\section{Further regularity of mild solutions}
Throughout this section, given $\alpha > 1$ fix $m_{0} \in \PP_{\alpha}(\R^{d})$ and $R \geq R(\alpha, [m_{0}]_{\alpha})$, where $R(\alpha, [m_{0}]_{\alpha})$ in defined as in Lemma \ref{convex}. At this point, we know that under assumptions {\bf (L1)}--{\bf (L3)} by Theorem \ref{existenceofmfg} there exists at least one Mean Field Games equilibrium $\eta \in \PP_{m_{0}}(\Gamma_{T}, R)$. Furthermore, if the Lagrangian $L$ is of the form \eqref{splittedlagrangian}, the coupling function $F$ and the terminal costs $G$ satisfy the strict monotonicity assumption, see Definition \ref{strictmonotonicity}, then the Mean Field Games equilibrium is unique. For this reasons, from now on we fix $R \geq R(\alpha, [m_{0}]_{\alpha})$.

Now, we are going to prove that any Mean Field Games equilibrium generates a family of probability measures $\{ m^{\eta}_{t}\}_{t \in [0,T]}$ which is $\frac{1}{2}$-H\"older continuous in time. Consequently, any mild solution $(V, m^{\eta})$ is such that the value function $V$ is locally Lipschitz continuous and locally fractionally semiconcave on $[0,T] \times \R^{d}$. Moreover, we will prove that there exists at least one Mean Field Games equilibrium $\eta \in \PP_{\alpha}(\Gamma_{T}, R)$ such that $t \to m^{\eta}_{t}$ is Lipschitz continuous.

Given the control system \eqref{dyn}, we have that the Hamiltonian associated with the Lagrangian function $L$ is defined as
\begin{equation*}
H(x,p,m)=\sup_{u \in \R^{k}} \Big\{-\langle p, Ax+Bu \rangle - L(x,u,m) \Big\}.
\end{equation*}

The Hamiltonian $H$ can be explicitly written as follows
\begin{equation}\label{hamiltonianfun}
H(x,p,m)=-\langle p, Ax \rangle + |B^{\star}p|^{2}-L(x, -B^{\star}p, m),
\end{equation}
for any $(x,p,m) \in \R^{d} \times \R^{k} \times \PP_{\alpha}(\R^{d})$. Moreover, it is easy to check that there exists a constant $c_{2} \geq 0 $ such that for any $(x,p,m) \in \R^{d} \times \R^{k} \times \PP_{\alpha}(\R^{d})$
\begin{equation}\label{hamiltoniangrowth}
|D_{p}H(x,p,m)| \leq c_{2} (1+ |x|+|p|).
\end{equation}


\subsection{Local Lipschitz continuity and local fractional semiconcavity of the Value function}

Let $(V, m^{\eta})$ a mild solution of the Mean Field Games problem associated with a Mean Field Games equilibrium $\eta \in \PP_{\alpha}(\Gamma_{T}, R)$. 

In this section, we prove that, given any equilibrium $\eta$, the associated measures $\{ m^{\eta}_{t}\}_{t \in [0,T]}$ are H\"older continuous and consequently, that the associated value function is locally semiconcave on $[0,T] \times \R^{d}$, linearly in space and with a fractional modulus of semiconcavity in time. Moreover, we show that, for any equilibrium $\eta \in \PP_{\alpha}(\Gamma_{T}, R)$, the value function $V$ is locally Lipschitz continuous on $[0,T] \times \R^{d}$.

We conclude this section proving that, under some extra assumptions on the data, there exists at least one equilibrium $\eta \in \PP_{\alpha}(\Gamma_{T}, R)$ such that $\{m^{\eta}_{t}\}_{t \in [0,T]}$ is Lipschitz continuous in time.

We recall that $V$ is defined as the value function 
\begin{equation*}
V(t,x)= \inf_{u \in L^{2}(0,T;\ \R^{k})} \left\{\int_{t}^{T}{L(\gamma(s; t,x,u), u(s), m^{\eta}_{s})\ ds} + G(\gamma(T; t,x,u), m^{\eta}_{T}) \right\}.
\end{equation*}

\begin{theorem}[{\bf H\"older continuity of equilibria}]\label{holder}
Given any Mean Field Games equilibrium $\eta$, the map $ t \to m^{\eta}_{t}$ is $\frac{1}{2}$-H\"older continuous in time.
\end{theorem}
\begin{proof}
By definition of $d_{1}$, we have that 
\begin{align*}
& d_{1}(m^{\eta}_{t}, m^{\eta}_{s}) = \inf_{\varphi \in \text{Lip}_{1}(\R^{d})} \int_{\R^{d}}{\varphi(x)(m^{\eta}_{t}-m^{\eta}_{s})(dx)}
\\
 = &  \inf_{\varphi \in \text{Lip}_{1}(\R^{d})}\int_{\Gamma_{T}}{(\varphi(\gamma(t))-\varphi(\gamma(s))\eta(d\gamma)} \leq  \int_{\Gamma_{T}}{|\gamma(t)-\gamma(s)|\eta(d\gamma)},
\end{align*}
where $\text{Lip}_{1}(\R^{d})$ is the set of Lipschitz continuous functions such that the Lipschitz constant is equal to $1$.

We recall that, since $\eta$ is a Mean Field Games equilibrium, we know that it is supported on the set of all minimizing curves of problem \eqref{min} and therefore, by Lemma \ref{equicontinuity} and recalling that $x=\gamma(0)$ we have that 
\begin{align*}
& d_{1}(m^{\eta}_{t}, m^{\eta}_{s}) \leq \int_{\Gamma_{T}}{|\gamma(t)-\gamma(s)|\eta(d\gamma)} 
\\
\leq & |t-s|^{\frac{1}{2}} \int_{\Gamma_{T}}{\left(\| A\|^{\frac{1}{2}} T^{\frac{1}{2}}\tilde{C}\big(1+|x|\big) + \| B\|^{\frac{1}{2}} K\right)\ \eta(d\gamma)}=\kappa([m_{0}]_{\alpha}])|t-s|^{\frac{1}{2}},
\end{align*}
where the constant $\kappa$ depends on the moment of $m_{0}$ which we know is bounded by construction.  
Thus, the proof is complete. 
\end{proof}

In order to prove the semiconcavity of the value function $V$, we need to add the following assumption on the Lagrangian $L$ and terminal cost $G$:
\begin{itemize}
\item[{\bf (L4)}] There exists two constants $w_{L} \geq 0 $ and $w_{G} \geq 0$ such that for any $\lambda \in [0,1]$, any radius $R>0$, any $u \in \R^{k}$, any $x_{0}$, $ x_{1}) \in B_{R}$, and any $m \in \PP_{1}(\R^{d})$ such that 
\begin{align*}
& \lambda L(x_{0}, u, m)+(1-\lambda)L(x_{1}, u, m) -L(\lambda x_{0} + (1-\lambda)x_{1}, u, m)  \leq w_{L}\lambda(1-\lambda) |x_{0}-x_{1}|^{2}, 
\\ \\
& \lambda G(x_{0}, m)+(1-\lambda)G(x_{1}, m) -G(\lambda x_{0} + (1-\lambda)x_{1}, m)  \leq w_{G}\lambda(1-\lambda) |x_{0}-x_{1}|^{2}.
\end{align*}
\end{itemize}

\begin{theorem}[{\bf Local fractional semiconcavity of $V$}]\label{Vsemiconcavity}
Let $R$ be a positive radius. Then, there exists a constant $\Lambda \geq 0$ such that for any $(t,x) \in [0,T] \times \overline{B}_{R}$, any $(h,\delta) \in \R \times \R$ such that $(x+h,t+\delta) \in [0,T] \times \overline{B}_{R}$ and $(x-h, t-\delta) \in [0,T] \times \overline{B}_{R}$ we have that
\begin{equation*} 
V(t+\delta, x+h)+V(t-\delta,x-h)-2V(t,x) \leq \Lambda\left(|h|^{2}+|\delta|^{\frac{3}{2}}\right).
\end{equation*}
\end{theorem}
\begin{proof}
We first prove that the value function $V$ is locally semiconcave in space uniformly in time and then, that it is locally semiconcave in space and time.

Let $R>0$ be a positive radius and fix $(t,x) \in [0,T] \times \overline{B}_{R}$. Let $h \in \R^{d}$ be such that $x+h$, $x-h \in \overline{B}_{R}$ and let $u^{*} \in L^{2}$ be an optimal control for $(t,x) \in [0,T] \times \overline{B}_{R}$. Then, define the following curves
\begin{align*}
\gamma(s)= & \gamma(s; t,x,u^{*}), \quad s \in [t,T]
\\
\gamma_{+}(s)= & \gamma(s; t,x+h,u^{*}), \quad s \in [t,T]
\\
\gamma_{-}(s)= & \gamma(s; t,x-h,u^{*}), \quad s \in [t,T].
\end{align*}
Thus, we have that 
\begin{align}\label{ssc}
\begin{split}
& V(t,x+h)+V(t,x-h)-2V(t,x) 
\\
\leq & \int_{t}^{T}{\Big(L(\gamma_{+}(s), u^{*}(s), m^{\eta}_{s}) + L(\gamma_{-}(s), u^{*}(s), m^{\eta}_{s}) - 2L(\gamma(s), u^{*}(s), m^{\eta}_{s}) \Big)\ ds} 
\\
+ & G(\gamma_{+}(T), m^{\eta}_{T}) + G(\gamma_{-}(T), m^{\eta}_{T}) - 2G(\gamma(T), m^{\eta}_{T}).
\end{split}
\end{align}
Consider, first, the expression involving only the terminal costs
\begin{align*}
& G(\gamma_{+}(T), m^{\eta}_{T}) + G(\gamma_{-}(T), m^{\eta}_{T}) - 2G(\gamma(T), m^{\eta}_{T}) 
\\
=\ & G(\gamma_{+}(T), m^{\eta}_{T}) + G(\gamma_{-}(T), m^{\eta}_{T}) -2G\left(\frac{\gamma_{+}(T)+\gamma_{-}(T)}{2}, m^{\eta}_{T}\right) 
\\
+\ & 2G\left(\frac{\gamma_{+}(T)+\gamma_{-}(T)}{2}, m^{\eta}_{T}\right) - 2G(\gamma(T), m^{\eta}_{T}).
\end{align*}
By assumptions {\bf (L1)} and {\bf (L4)} we deduce that
\begin{align*}
& G(\gamma_{+}(T), m^{\eta}_{T}) + G\left(\gamma_{-}(T), m^{\eta}_{T}\right) -2G\left(\frac{\gamma_{+}(T)+\gamma_{-}(T)}{2}, m^{\eta}_{T}\right) \leq w_{G} |\gamma_{+}(T)-\gamma_{-}(T)|^{2},
\\
& 2G\left(\frac{\gamma_{+}(T)+\gamma_{-}(T)}{2}, m^{\eta}_{T}\right) - 2G\left(\gamma(T), m^{\eta}_{T}\right) \leq \| G \|_{\infty} |\gamma_{+}(T)+\gamma_{-}(T)-2\gamma(T)|.
\end{align*}
By the definition of $\gamma$, $\gamma_{+}$ and $\gamma_{-}$ we have that these curves are solutions of \eqref{dyn}. Therefore, we get that there exists a real positive constant $W$ such that
\begin{align*}
& |\gamma_{+}(T)-\gamma_{-}(T)|^{2} \leq W|h|^{2},
\\
&  |\gamma_{+}(T)+\gamma_{-}(T)-2\gamma(T)| \leq W|h|^{2}. 
\end{align*}
Hence, we deduce that
\begin{equation*}
G(\gamma_{+}(T), m^{\eta}_{T}) + G(\gamma_{-}(T), m^{\eta}_{T}) - 2G(\gamma(T), m^{\eta}_{T}) \leq W\left(w_{G}+\|G\|_{\infty} \right)|h|^{2}.
\end{equation*}
By almost similar arguments, one can prove that also the integral term in \eqref{ssc} is bounded by a constant times $|h|^{2}$. This proves that $V$ is locally semiconcave in space uniformly in time.

We prove now that $V$ is locally semiconcave on $[0,T] \times \R^{d}$.
Fix $(t,x) \in [0,T] \times \overline{B}_{R}$ and let $h \in \R^{d}$ and $\delta \in \R$ be such that $x+h$, $x-h \in \overline{B}_{R}$ and $0 < t-\delta < t+\delta<T$. Let $u^{*}$ be an optimal control for $(t,x)$ and define the following control function in $L^{2}$
\begin{equation*}
\bar{u}(s)=u^{*}\left(\frac{t+\delta+s}{2} \right), \quad s \in [t-\delta, t+\delta].
\end{equation*}
By the Dynamic Programming Principle \eqref{DPP}, we get that
\begin{align*}
& V(t+\delta, x+h)+V(t-\delta, x-h)-2V(t,x) 
\\
\leq\ & \underbrace{V(t+\delta, x+h) + V(t+\delta, \gamma(t+\delta; t-\delta, x-h, \bar{u})) -2V(t+\delta, \gamma(t+\delta;t,x,u^{*}) )}_{I}
\\
+& \underbrace{\int_{t-\delta}^{t+\delta}{L(\gamma(s; t-\delta,x-h,u^{*}), \bar{u}(s), m^{\eta}_{s})\ ds}-2\int_{t}^{t+\delta}{L(\gamma(s;t,x,u^{*}), u^{*}(s), m^{\eta}_{s})\ ds}}_{II}.
\end{align*}
Thus, by the first parte of the proof term $I$ is bounded by a constant times $|h|^{2}+|\delta|^{2}$. Now, we have to estimate term $II$. Let us denote, for simplicity, by $\gamma^{-}$ the curve $\gamma(\cdot\ ; t-\delta, x-h, u^{*})$. Then, by assumption {\bf (L1)} we have that there exists a constant $D \geq 0 $ such that
\begin{align}\label{semi1}
\begin{split}
II =\ & 2\int_{t}^{t+\delta}{\Big(L(\gamma^{-}(2s-t-\delta), u^{*}(s), m^{\eta}_{2s-t-\delta})-L(\gamma(s), u^{*}(s), m^{\eta}_{s})\Big)\ ds}
\\
\leq & D \int_{t}^{t+\delta}{\Big( |\gamma^{-}(2s-t-\delta)-\gamma(s)| + d_{1}(m^{\eta}_{2s-t-\delta}, m^{\eta}_{s})\Big)\ ds}
\end{split}
\end{align}
Since $\eta$ is a Mean Field Games equilibrium we know by Theorem \ref{holder} that the generated measure $\{m^{\eta}_{t}\}_{t \in [0,T]}$ is $\frac{1}{2}$-H\"older continuous in time with respect to the $d_{1}$ distance. Therefore,
\begin{align}\label{semi2}
\begin{split}
& \int_{t}^{t+\delta}{d_{1}(m^{\eta}_{2s-t-\delta}, m^{\eta}_{s})\ ds}\leq \kappa([m_{0}]_{\alpha}) \int_{t}^{t+\delta}{|s-t-\delta|^{\frac{1}{2}}\ ds}
\leq \frac{2}{3}\kappa([m_{0}]_{1})|\delta|^{\frac{3}{2}}.
\end{split}
\end{align}
Now, we have to estimate the distance between the curves $\gamma_{-}$ and $\gamma$. For that, we recall that since $\gamma_{-}$ and $\gamma$ are solutions of \eqref{dyn} we know that
\begin{align*}
\gamma^{-}(2s-t-\delta)=& e^{(s-t+\delta)A}(x-h) + \int_{t-\delta}^{2s-t-\delta}{e^{(\tau-t+\delta)A}B\bar{u}(\tau)\ d\tau},
\\
\gamma(s)=& e^{(s-t)A}x+\int_{t}^{s}{e^{(\tau-t)A}Bu^{*}(\tau)\ d\tau}.
\end{align*}

By \cite[Theorem 7.4.6]{bib:SC}, without loss of generality, we can assume that $u^{*}$ belongs to $L^{\infty}$ and consequently, $\bar{u} \in L^{\infty}$. Thus, we obtain that for any $s \in [t, t+\delta]$
\begin{align*}
|\gamma^{-}(2s-t-\delta)-\gamma_{-}(s)| \leq\ & e^{T\| A\|}|h|
+ 2se^{T\|A\|}\|B\| \|\bar{u}\|_{\infty}+(s-t)e^{T\|A\|}\|B\|\|u^{*}\|_{\infty}.
\end{align*}
Therefore, we deduce that
\begin{align}\label{semi3}
\begin{split}
& \int_{t}^{t+\delta}{\left( e^{T\| A\|}|h|
+ 2se^{T\|A\|}\|B\| \|\bar{u}\|_{\infty}+(s-t)e^{T\|A\|}\|B\|\|u^{*}\|_{\infty} \right)\ ds}
\\
\leq &  \delta e^{T\| A\|}|h| + \Big(2e^{T\|A\|}\|B\| \|\bar{u}\|_{\infty}+ e^{T\|A\|}\|B\|\|u^{*}\|_{\infty}\Big)\delta^{2}.
\end{split}
\end{align}
Hence, plugging inequalities \eqref{semi2} and \eqref{semi3} into \eqref{semi1} the proof is complete.
\end{proof}

\begin{remark}
We note that Theorem \ref{Vsemiconcavity} guarantees that the function $x \mapsto V(t,x)$ is linearly semiconcave, locally uniformly in time.
\end{remark}

The proof of the following theorem is given in Appendix \ref{appendix1} since the techniques we have used to prove it are classical in optimal control theory.

\begin{theorem}\label{Vlipschitz}
$V$ is locally Lipschitz continuous on $[0,T] \times \R^{d}$.
\end{theorem}

\subsection{Lipschitz regularity of Mean Field Games equilibrium}

Define the following class of curves on $\PP_{\alpha}(\R^{d})$
\begin{equation*}
\text{Lip}(\PP_{\alpha}) = \left\{t\in [0,T] \mapsto m_{t} \in \PP_{\alpha}(\R^{d}):\ \sup_{\substack{t \not= s \\ t,s \in [0,T]}} \frac{d_{1}(m_{t}, m_{s})}{|t-s|} < \infty \right\},
\end{equation*}
and define 
\begin{equation*}
\PP_{m_{0}}^{\text{Lip}(\PP_{\alpha})}(\Gamma_{T}) = \Big\{\eta \in \PP_{m_{0}}(\Gamma_{T}, R): m^{\eta}_{t} \in \text{Lip}(\PP_{1}) \Big\}.
\end{equation*}

\begin{remark}
\em{The set $\PP_{m_{0}}^{\text{Lip}(\PP_{\alpha})}(\Gamma_{T})$ is non-empty. Following the construction we have done in Remark \ref{empty}, let $p: \R^{d} \to \Gamma_{T}$ be defined as
\begin{equation*}
x \mapsto p[x](t):=e^{tA}x, \forall\ t \in [0,T]
\end{equation*}
and define $\eta = p \sharp m_{0}$. Therefore, by Remark \ref{empty}, we only need to prove that $m^{\eta} \in \text{Lip}(\PP_{\alpha})$.

Indeed,
\begin{align*}
d_{1}(m^{\eta}_{t_{1}} , m^{\eta}_{t_{2}}) = & \sup_{\phi \in 1-\text{Lip}}\int_{\R^{d}}{\phi(x)\big(m^{\eta}_{t_{1}}(dx)-m^{\eta}_{t_{2}}(dx) \big)}
\\
= & \sup_{\phi \in 1-\text{Lip}} \int_{\Gamma_{T}}{\big(\phi(\gamma(t_{1}))-\phi(\gamma(t_{2})) \big) \eta(d\gamma)}
\\
=& \sup_{\phi \in 1-\text{Lip}} \int_{\Gamma_{T}}{\big(\phi(\gamma(t_{1}))-\phi(\gamma(t_{2})) \big) p \sharp m_{0} (d\gamma)}
\\
=& \sup_{\phi \in 1-\text{Lip}} \int_{\R^{d}}{\big(\phi(p[x](t_{1}))-\phi(p[x](t_{2})) \big) m_{0}(dx)}
\\
=& \sup_{\phi \in 1-\text{Lip}} \int_{\R^{d}}{\big(\phi(e^{At_{1}}x)-\phi(e^{At_{2}}x) \big) m_{0}(dx)}
\\ 
\leq & \int_{\R^{d}}{\big|e^{At_{1}}x-e^{At_{2}}x \big| m_{0}(dx)}.
\end{align*}
Since the function $t \mapsto e^{At}x$ is Lipschitz continuous in any compact subintervals of $\R$ we get the conclusion.
\begin{equation*}
\eqno{\square}
\end{equation*}
}
\end{remark}

\begin{proposition}\label{boundedvelocity}
Assume that the Hamiltonian $H$ satisfy the following
\begin{itemize}
\item[{\bf (H1)}] there exists a constant $c_{3} > 0$ such that for any $(x,p,m) \in \R^{d} \times \R^{k} \times \PP_{\alpha}(\R^{d})$
\begin{equation*}
\langle D_{x}H(x,p,m), p \rangle \geq c_{3}|p|^{2}-c_{4}.
\end{equation*}
\end{itemize}
Fix $x \in \R^{d}$ and $\eta \in \PP_{m_{0}}^{\text{Lip}(\PP_{1})}(\Gamma_{T})$. Let $u^{*}$ be an optimal control for the problem \eqref{min} and let $\gamma^{*}$ be the minimizing curve generated by $u$. Then, there exists a real positive constant $Q_{1}$ such that 
\begin{equation*}
\| \dot\gamma^{*}\|_{\infty} \leq Q_{1}(1+|x|).
\end{equation*}
\end{proposition}
\medskip
\begin{remark}
\em Note that, by construction and the explicit form of the Hamiltonian $H$ given in \eqref{hamiltonianfun}, assumption {\bf (H1)} can be restated in terms of the Lagrangian $L$ as follows
\begin{equation*}
\langle D_{x}L(x, -B^{\star}p, m), p \rangle \leq -c_{3}|p|^{2}-c_{4}, \quad \forall\ (x,p,m) \in \R^{d} \times \R^{k} \times \PP_{\alpha}(\R^{d}).
\end{equation*}
\end{remark}
\medskip
\begin{proof}
Since $\eta \in \PP_{m_{0}}^{\mathcal{L}_{b}}(\Gamma_{T})$, by the maximum principle in Hamiltonian form, Theorem \ref{Hmaxprinciple}, we have that there exists an arc $p^{*}: [0,T] \to \R^{d}$ such that
\begin{align*}
\begin{cases}
\dot\gamma^{*}(t)=-D_{p}H(\gamma^{*}(t), p^{*}(t), m^{\eta}_{t}),
\\
\dot p^{*}(t)= D_{x}H(\gamma^{*}(t), p^{*}(t), m^{\eta}_{t}),
\end{cases}
\end{align*}
and by the transversality condition in Theorem \ref{maxprinciple},  
\begin{equation}\label{trasv}
	p^{*}(T)=D_{x}G(\gamma^{*}(T), m^{\eta}_{T}).
\end{equation}

Thus, by \eqref{hamiltoniangrowth} we obtain that
\begin{align*}
& \sup_{t \in [0,T]} |\dot\gamma^{*}(t)| = \sup_{t \in [0,T]} |D_{p}H(\gamma^{*}(t), p^{*}(t), m^{\eta}_{t})| 
\\
\leq &  c_{2} \left(1+\sup_{t \in [0,T]} |\gamma^{*}(t)|+\sup_{t \in [0,T]} |p^{*}(t)| \right) \leq \beta \left(1+ |x|+\sup_{t \in [0,T]} |p^{*}(t)|\right ),
\end{align*}
where the last inequality follows by Corollary \ref{infboundmin} for some constant $\beta \geq 0$.
Thus, we have reduced the problem to prove that the dual arc $p^{*}$ is bounded.

By Theorem \ref{Hmaxprinciple}, we have that
\begin{equation*}
\dot p^{*}(t)= D_{x}H(\gamma^{*}(t), p^{*}(t), m^{\eta}_{t})
\end{equation*}
and, by assumption {\bf (H1)} we deduce that
\begin{equation*}
\frac{d}{dt} \left(\frac{1}{2} |p^{*}(t)|^{2} \right) = \langle D_{x}H(\gamma^{*}(t), p^{*}(t), m^{\eta}_{t}), p^{*}(t) \rangle \geq c_{3} |p^{*}(t)|^{2}-c_{4}.
\end{equation*}
Therefore, we get
\begin{equation*}
\frac{d}{dt}\left(\frac{1}{2}e^{-2c_{3}t}|p^{*}(t)|^{2} \right) \geq -c_{4}e^{-2c_{3}t}	
\end{equation*}
and integrating over $[t,T]$ both side of the inequality we obtain 
\begin{equation*}
|p^{*}(t)|^{2} \leq  |p^{*}(T)|^{2}+\frac{c_{4}}{2c_{3}}
\end{equation*}
which is bounded by \eqref{trasv} and assumptions on $G$.
\end{proof}

We recall the definition of the set-valued map $E$ given in the Section 3, that is
\begin{equation*}
E: \big(\PP_{m_{0}}(\Gamma_{T}, R), d_{1} \big) \rightrightarrows \big(  \PP_{m_{0}}(\Gamma_{T}, R), d_{1}  \big)
\end{equation*}
such that
\begin{equation*}
\eta \mapsto E(\eta) = \big\{\nu \in \mathcal{P}_{m_{0}}(\Gamma_{T}, R):\  \supp(\nu_{x}) \subset\ \Gamma^{*}_{\eta}(x),\ m_{0}-\text{a.e.} \big\}.
\end{equation*}

\medskip
\begin{lemma}\label{inc}
$E(\PP_{m_{0}}^{\text{Lip}(\PP_{\alpha})}) \subset \PP_{m_{0}}^{\text{Lip}(\PP_{\alpha})}$.
\end{lemma}
\begin{proof}
Fix $\eta \in \PP_{m_{0}}^{\text{Lip}(\PP_{\alpha})}$ and let $\mu$ be a Borel probability measure in $E(\eta)$. We want to prove that for any $t_{1}$, $t_{2} \in [0,T]$, with $t_{1} < t_{2}$ 
\begin{equation*}
\sup_{\substack{t \not= s \\ t,s \in [0,T]}} \frac{d_{1}(m_{t}, m_{s})}{|t-s|} < \infty.
\end{equation*} 
Hence
\begin{align*}
d_{1}(m^{\mu}_{t_{1}}, m^{\mu}_{t_{2}}) = & \sup_{\phi \in \text{Lip}_{1}(\R^{d})} \int_{\R^{d}}{\phi(x) \big(m^{\mu}_{t_{1}}(dx) - m^{\mu}_{t_{1}}(dx)\big)}
\\
= & \sup_{\phi \in \text{Lip}_{1}(\R^{d})} \int_{\Gamma_{T}}{\big(\phi(\gamma(t_{1}))-\phi(\gamma(t_{2}) \big)\ \mu(d\gamma)} 
\\
\leq & \int_{\Gamma_{T}}{\big|\gamma(t_{1})-\gamma(t_{2}) \big|\ \mu(d\gamma)} \leq |t_{1}-t_{2}|\int_{\Gamma_{T}}{\|\dot\gamma \|_{\infty}|\ \mu(d\gamma)}
\\
\leq & |t_{1}-t_{2}|\int_{\Gamma_{T}}{Q_{1}(1+|x|)\ \mu(d\gamma)},
\end{align*}
where the last inequality follows by Proposition \ref{boundedvelocity}. Therefore, observing that $x=\gamma(0)$  and $\mu$ belongs to $\PP_{m_{0}}(\Gamma_{T}, R)$, we obtain the conclusion.
\end{proof}

\begin{theorem}[{\bf Existence of Lipschitz Mean Field Games equilibria}]\label{lipschitzmfg}
There exist at least one Mean Field Games equilibrium such that the associated family of measure $\{ m^{\eta}_{t}\}_{t \in [0,T]}$ belongs to $\text{Lip}(\PP_{\alpha})$.
\end{theorem}
\begin{proof}
It is sufficient to prove that the set-valued map $E: \PP_{m_{0}}^{\text{Lip}(\PP_{\alpha})}(\Gamma_{T}) \rightrightarrows \PP_{m_{0}}^{\text{Lip}(\PP_{\alpha})}$ has a fix point and in order to prove it we want to use Kakutani's fixed point theorem.

We recall that by Theorem \ref{mfgclosedgraph} we have that the map $E$ has closed graph and so also the restriction of $E$ on $\PP_{m_{0}}^{\text{Lip}(\PP_{\alpha})}$. Moreover, since $\PP_{m_{0}}^{\text{Lip}(\PP_{\alpha})} \subset \PP_{m_{0}}(\Gamma_{T}, R)$ we have that $\PP_{m_{0}}^{\text{Lip}(\PP_{\alpha})}$ is compact.

Therefore, all the assumptions of Kakutani's fixed point theorem are satisfied and this concludes the proof.
\end{proof}

\begin{corollary}\label{linearsemiconcavity}
Let $\eta \in \PP_{\alpha}(\Gamma_{T}, R)$ be a Lipschitz Mean Field Games equilibrium and let $(V, m^{\eta})$ be a mild solution associated with $\eta$. Then, the value function $V$ is locally semiconcave on $[0,T] \times \R^d$ with a linear modulus of semiconcavity. 
\end{corollary}

\section{Mean Field Games: PDEs system}
\subsection{Optimal syntesis}

In order to deduce the PDE system for our Mean Field Games problem, we have to derive first some optimality conditions for the following problem:
	\begin{align*}
\tag{{\bf OC}}
 J(x,u)=\inf_{\gamma \in \Gamma_{T}(x)} \left\{g(\gamma(T))+\int_{0}^{T}{L(t, \gamma(t), u(t))\ dt}\right\}.
\end{align*}
	As usual, let $V$ be the value function of the above {\bf (OC)} problem.

	Let $p_{0}$ be a point in $D^{*}_{x}V(t_{0}, x_{0})$ such that $(t_{0}, x_{0}) \in [0,T] \times \overline{B}_{R}$. By definition of reachable gradient, there exists a sequence $\{ x_{k}\}_{k \in \N}$ such that
	\begin{align*}
	x_{k} & \to x_{0}
	\\
	-p_{0}& =\lim_{k \to \infty} D_{x}V(t_{0}, x_{k}).
	\end{align*}

Let $\bar{u}_{k}$ and $\bar{\gamma}_{k}$ be, respectively, an optimal control and an optimal trajectory with starting point $(t_{0}, x_{k})$. By the maximum principle (Theorem \ref{maxprinciple}), we have that there exists an absolutely continuous arc $\bar{p}_{k}$ such that
\begin{align}\label{dualarcseq}
\begin{cases}
& -\dot{\bar{p}}_{k}(t)=A^{*}\bar{p}_{k}(t)+D_{x}L(t, \bar{\gamma}_{k}(t), \bar{u}_{k}(t))
\\
& \bar{p}_{k}(T)=Dg(\bar{\gamma}_{k}(T)).
\end{cases}
\end{align}
By the maximum principle in Hamiltonian form (Theorem \ref{Hmaxprinciple})
\begin{align}\label{maxpr1}
\begin{cases}
&\dot{\bar{\gamma}}_{k}(t)=-D_{p}H(t, \bar{\gamma}_{k}(t), \bar{p}_{k}(t))
\\
& \dot{\bar{p}}_{k}(t)=D_{x}H(t, \bar{\gamma}_{k}(t), \bar{p}_{k}(t)).
\end{cases}
\end{align}

	Since the sequence $\{ x_{k}\}_{k \in \N}$ is convergent, by Corollary \ref{infboundmin} and Proposition \ref{boundedvelocity} we obtain that $\{ \gamma_{k}\}_{k \in \N}$ is equibounded and equicontinuous.

	Moreover, by \eqref{dualarcseq} we have that for any $t \geq t_{0}$
	\begin{equation*}
	\bar{p}_{k}(t)=e^{(T-t)A^{*}}Dg(\bar{\gamma}_{k}(T))+\int_{t}^{T}{e^{(s-t)A^{*}}D_{x}L(s, \bar\gamma_{k}(s), \bar{u}_{k}(s))\ ds}.
	\end{equation*}
Thus, it easily follows that also the sequence of dual arcs $\{ \bar{p}_{k}\}_{k \in \N}$ is equibounded and equicontinuous. Therefore, there exist an absolutely continuous arc $\bar{p}$ and a curve $\bar\gamma$ such that $\bar{p}_{k} \to \bar{p}$ and $\bar\gamma_{k} \to \bar\gamma$, uniformly as $k \to \infty$.
	
	Since $L$ is a strict Tonelli Lagrangian, see Definition \ref{def2}, we have that there exists a constant $\kappa \geq 0$ such that
	\begin{equation*}
	|D_{x}L(t,x,u)| \leq \kappa(1+|u|^{2}).
	\end{equation*}
	Moreover, since $x \in \overline{B}_{R}$ we deduce by \cite[Theorem 7.4.6]{bib:SC} that there exists a constant $\tilde{\kappa} \geq 0$ such that $\| u_{k} \|_{\infty} \leq \tilde{\kappa}$. Consequently, we obtain that $D_{x}L(t, \bar\gamma_{k}(t), \bar{u}_{k}(t))$ weakly converges in $L^{2}(0,T; \R^{d})$ to $D_{x}L(t, \bar\gamma(t), \bar{u})$ as $k \to \infty$.

	Therefore, passing to the limit in \eqref{dualarcseq} we get that $\bar{p}$ is a solution of the limit equation and by the maximum principle the pair $(\bar{\gamma}, \bar{p})$ solves system \eqref{maxpr1}.
	In conclusion, as $k \to \infty$ in the value function we obtain that the curve $\bar\gamma$ is a minimizer for $(t_{0}, x_{0})$.

\subsection{Weak solutions}

In this section, we consider the case of splitted Langrangian, that is $L$ is of the form \eqref{splittedlagrangian}. 

We recall that, given the control system \eqref{dyn}, the Hamiltonian associated with the Lagrangian function $L$ is defined as
\begin{equation*}
H(x,p)=\sup_{u \in \R^{k}} \Big\{-\langle p, Ax+Bu \rangle - \ell(x,u) \Big\}.
\end{equation*}

For $\alpha > 1$, let $m_{0} \in \PP_{\alpha}(\R^{d})$ be a Borel probability measure and introduce the following Mean Field Games PDEs system 
\begin{align}\label{PDEsystem}
\begin{cases}
-\partial_{t} V(t,x) + H(x, D_{x}V(t,x))=F(x,m_{t}), \quad & (t,x) \in [0,T] \times \R^{d}
\\
\partial_{t}m_{t}+\ddiv\Big(m_{t}D_{p}H(x, D_{x}V(t,x)) \Big)=0, \quad & (t,x) \in [0,T] \times \R^{d}
\\
m_{0}=m_{0}, \quad V(T,x)=G(x, m_{T}),\ \forall\ x \in \R^{d}.
\end{cases}
\end{align}

\begin{definition}[{\bf Weak solutions}]
We say that $(V, m) \in W^{1,\infty}([0,T] \times \mathbb{R}^{d}) \times C([0,T], \mathcal{P}_{\alpha}(\mathbb{R}^{d}))$ is a weak solution of the Mean Field Games PDEs system if:
\begin{itemize}
\item[($i$)] $m$ is a solution in the sense of distribution of the continuity equation, i.e. for any test function $\varphi \in C^{1}_{c}([0,T) \times \R^{d})$ we have that
\begin{equation*}
-\int_{\R^{d}}{\varphi(0,x)\ m_{0}(dx)}= \int_{0}^{T}\int_{\R^{d}}{\Big(\partial_{t} \varphi(t,x)-\langle D_{x}\varphi(t,x), D_{p}H(x, D_{x}V(t,x)) \rangle \Big)\ m_{t}(dx)}.
\end{equation*}
\item[($ii$)] $V$ is a continuous viscosity solution of Hamilton-Jacobi equation.
\end{itemize}
\end{definition}

\begin{remark}\em\label{remark1}
We recall that by classical optimal control theory, see for instance \cite{bib:SC}, the following holds:
\begin{enumerate}
	\item from the maximum principle one can deduce that any minimizer $\gamma$ of problem \eqref{min} has the same regularity of the data, thus in this case we obtain that $\gamma \in C^{2}$, see Theorem \ref{Hmaxprinciple};
	\item given a Mean Field Games equilibrium $\eta$ we have that for any $x \in \supp(m^{\eta}_{t})$ the value function $V$ is differentiable since the value function of an optimal control problem with a strictly convex Hamiltonian (with respect to $p$) is known to be differentiable in the interior of any optimal trajectory, see for instance \cite[Theorem 6.4.7]{bib:SC} and \cite[Proposition 4.4]{bib:CH}.
	\end{enumerate}
\end{remark}

\begin{theorem}[{\bf Equivalence between mild and weak solutions}]\label{weaksolutions}
Assume {\bf (L1)}---{\bf (L4)} and {\bf (H1)}. Fix $\alpha > 1$ and let $m_{0} \in \PP_{\alpha}(\R^{d})$ be an absolutely continuous with respect the Lebesgue measure and with compact support. Then, $(V,m) \in C([0,T] \times \mathbb{R}^{d}) \times C([0,T], \mathcal{P}_{\alpha}(\mathbb{R}^{d}))$ is a mild solution of the Mean Field Games problem if and only if it is a weak solution of system \eqref{PDEsystem}.
\end{theorem}
\begin{proof}
First, we show that any mild solutions $(V, m^{\eta})$ is a weak solution.

Let $V$ be the value function defined as in Definition \ref{mildsolution}, in expression \eqref{mfgvaluefunction}. Then, it is well-known that it is a continuous viscosity solution of the Hamilton-Jacobi equation in system \eqref{PDEsystem} and satisfies the terminal condition. Hence, we are left to prove that $m^{\eta}$ is a solution of the continuity equation in system \eqref{PDEsystem} in the sense of distributions.

Indeed, for any $\varphi \in C^{1}_{c}([0,T) \times \R^{d})$, we have that
\begin{align*}
 \frac{d}{dt}& \int_{\R^{d}}{\varphi(t,x)\ m^{\eta}_{t}(dx)}= \frac{d}{dt}\int_{\Gamma_{T}}{\varphi(t,\gamma(t))\ \eta(d\gamma)}
\\
=& \int_{\Gamma_{T}}{\Big(\partial_{t} \varphi(t,\gamma(t))+\langle D_{x}\varphi(t,\gamma(t)), \dot\gamma(t) \rangle \Big)\ \eta(d\gamma)},
\end{align*}
where the last integral is well-posed by point (1) in Remark \ref{remark1}.
Since $\eta \in \PP_{m_{0}}(\Gamma_{T}, R)$ is a Lipschitz Mean Field Games equilibrium we know that $\eta$ is supported on the minimizers of problem \eqref{min}. So, by Theorem \ref{Hmaxprinciple} we know that
\begin{equation*}
\dot\gamma(t)=-D_{p}H\big(\gamma(t), D_{x}V(t,\gamma(t))\big).
\end{equation*}
Therefore, 
\begin{align*}
 \frac{d}{dt}& \int_{\R^{d}}{\varphi(t,x)\ m^{\eta}_{t}(dx)} 
\\
= & \int_{\Gamma_{T}}{\Big(\partial_{t} \varphi(t,\gamma(t))+\langle D_{x}\varphi(t,\gamma(t)), \dot\gamma(t) \rangle \Big)\ \eta(d\gamma)}
\\
= & \int_{\Gamma_{T}}{\Big(\partial_{t} \varphi(t,\gamma(t))-\langle D_{x}\varphi(t,\gamma(t)), D_{p}H(\gamma(t), D_{x}V(t,\gamma(t)))  \rangle \Big)\ \eta(d\gamma)}
\\
=& \int_{\R^{d}}{\Big(\partial_{t} \varphi(t,x)-\langle D_{x}\varphi(t,x), D_{p}H(x, D_{x}V(t,x))  \rangle \Big)\ m^{\eta}_{t}(dx)},
\end{align*}
where the last integral in above series of equality is well-posed by point (2) in Remark \ref{remark1}. 
The conclusion follows by integrating the above equalities over $[0,T]$.

Now, let $(V,m)$ be a weak solution of Mean Field Games system. Since $V$ is a viscosity solution of the Hamilton-Jacobi equation we know that it can be represented by the formula \eqref{mfgvaluefunction} in Definition \ref{mildsolution}. Hence, we only have to prove that there exists a Mean Field Games equilibrium $\eta$ such that $m_{t}=e_{t} \sharp \eta$. 

Since $m$ is a solution of the continuity equation in the sense of distributions, by the superposition principle \cite[Theorem 8.2.1]{bib:AGS} we know that there exists a probability measure $\mu \in \PP(\Gamma_{T})$ such that $m_{t}= e_{t} \sharp \mu$ and $\mu$-a.e. is a solution of the following equation
\begin{equation}\label{feedback}
\dot\gamma(t)=-D_{p}H(\gamma(t), D_{x}V(t,\gamma(t))), \quad t \in [0,T].
\end{equation}
As $m_{0}=e_{0} \sharp \mu$, by Theorem \ref{disintegration} there exists a family of Borel probability measures $\mu_{x}$, for any $x \in \supp(m_{0})$, such that
\begin{equation*}
\mu(d\gamma)=\int_{\R^{d}}{\mu_{x}(d\gamma)m_{0}(x)\ dx}.
\end{equation*}
Since $m_{0}$ is absolutely continuous with compact support and the value function $V$ is locally Lipschitz continuous, it follows that $m_{0}$-a.e. and $\mu_{x}$-a.e. $\gamma$ is a solution of \eqref{feedback} such that $\gamma(0)=x$. Therefore, by the optimal synthesis explained above, such a curve $\gamma$ is a minimizer of the underlying optimal control problem. Hence, the measures $\mu_{x}$ are supported on minimizing curves of the optimal control problem. Consequently, $\mu$ is a Mean Field Games equilibrium for $m_{0}$. 

\end{proof}

The following result is an immediate consequence of Theorem \ref{weaksolutions} and Theorem \ref{uniquenessmfg}.

\begin{corollary}
Assume that $F$ is strictly monotone, in the sense of definition \ref{strictmonotonicity}. Let $\eta_{1}$, $\eta_{2} \in \PP_{m_{0}}(\Gamma_{T}, R)$ be two Lipschitz Mean Field Games equilibria and let $(V_{1}, m^{\eta_{1}})$, $(V_{2}, m^{\eta_{2}})$ be, respectively, the weak solutions of system \eqref{PDEsystem}. Then, $V_{1} \equiv V_{2}$.
\end{corollary}


\appendix

\section{Proof of Theorem \ref{Vlipschitz}}\label{appendix1}

We divide the proof in two steps: first, we prove that $V$ is locally Lipschitz in space and then, we prove that it is locally Lipschitz in both the variables.

Let $R$ be a positive radius and denote by $B_{R}$ the ball of radius $R$ centered in the origin on $\R^{d}$. Fix $x \in \overline{B}_{R}$ and $h \in \R^{d}$ such that $x+h \in \overline{B}_{R}$. Then, given an  optimal control $u^{*}$ associated with $(t,x) \in [0,T] \times \overline{B}_{R}$ we get that
\begin{align}\label{lip1}
\begin{split}
V(t,x+h)-V(t,x) \leq & \int_{t}^{T}{\big(L(\gamma(s; t, x+h, u^{*}), u^{*}(s), m^{\eta}_{s}) - L(\gamma(s; t ,x,u^{*}), u^{*}(s), m^{\eta}_{s}) \big)\ ds} 
\\
+& G(\gamma(T; t,x+h,u^{*}),m^{\eta}_{T}) - G(\gamma(T; t,x,u^{*}), m^{\eta}_{T}). 
\end{split}
\end{align}
Thus, we have to estimate the distance between two admissible paths: the one starting in $(t,x)$ and the other one starting in $(t,x+h)$. Recall that 
\begin{equation*}
\gamma(s;t,x,u)=e^{(s-t)A}x+\int_{t}^{s}{e^{(\tau-t)A}Bu^{*}(\tau)\ d\tau}, \quad \forall\ s \in [t,T]
\end{equation*}
to obtain
\begin{equation*}
|\gamma(s;t,x+h,u^{*})-\gamma(s;t,x,u^{*})|	\leq e^{T\| A\|}|h|, \quad \forall\ s \in [t,T].
\end{equation*}
Therefore, by assumption {\bf (L2)} we get 
\begin{equation*}
G(\gamma(T; t,x+h,u^{*}),m^{\eta}_{T}) - G(\gamma(T; t,x,u^{*}), m^{\eta}_{T}) \leq \| G \|_{\infty} e^{T\| A \|} |h|.
\end{equation*}
So, we just have to bound the integral term in \eqref{lip1}. By assumption {\bf (L3)}, we have that
\begin{align*}
& \int_{t}^{T}{\big(L(\gamma(s; t, x+h, u^{*}), u^{*}(s), m^{\eta}_{s}) - L(\gamma(s; t ,x,u^{*}), u^{*}(s), m^{\eta}_{s}) \big)\ ds} 
\\
= & \int_{t}^{T} \int_{0}^{1}{\langle D_{x}L(\lambda \gamma(s;t,x+h,u^{*}) + (1-\lambda) \gamma(s; t,x,u^{*}), u^{*}(s), m^{\eta}_{s}, \gamma(s; t,x+h,u^{*})-\gamma(s;t,x,u^{*}) \rangle\ ds} 
\\
\leq &  \int_{t}^{T} \int_{0}^{1}{\big| D_{x}L(\lambda \gamma(s;t,x+h,u^{*}) + (1-\lambda) \gamma(s; t,x,u^{*})\big| \big|u^{*}(s), m^{\eta}_{s}, \gamma(s; t,x+h,u^{*})-\gamma(s;t,x,u^{*})\big|\ ds} 
\\
\leq & \int_{t}^{T}\int_{0}^{1}{c_{2}\big(1+|u^{*}(s)| \big) \big|\gamma(s; t,x+h,u^{*})-\gamma(s;t,x,u^{*})\big|\ ds}
\\
\leq &\ Tc_{2}e^{T\| A \|} |h| + c_{2}\sqrt{T} \| u^{*} \|_{2} |h|= \left(c_{2}Te^{T\| A \|} + c_{2}\sqrt{T}K \right) |h|,
\end{align*}
where $\| u^{*} \|_{2} \leq K$ by Proposition \ref{bound}. Then, we conclude that
\begin{equation*}
V(t,x+h)-V(t,x) \leq \left(c_{2}Te^{T\| A \|} + c_{2}\sqrt{T}K+ \| G \|_{\infty}e^{T\|A\|} \right) |h|.
\end{equation*}

By similar considerations, one can easily prove that the reverse inequality   also holds true. Therefore, we have that $V$ is locally Lipschitz in space.

We now prove  that $V$ is locally Lipschitz in space and time on $[0,T] \times \overline{B}_{R}$ for any $R>0$. Fix $t \in [0,T]$, $x \in \overline{B}_{R}$ and let $\delta \in \R$ be such that $t+\delta \in [0,T]$.

We recall that, by the Dynamic Programming Principle \eqref{DPP},
\begin{equation}\label{DPPmfg}
V(t,x) = \inf_{u \in L^{2}} \left\{ V(t+\delta, \gamma(t+\delta; t,x,u)) + \int_{t}^{t+\delta}{L(\gamma(s; t,x,u), u(s), m^{\eta}_{s})\ ds} \right\}.
\end{equation}
Moreover, by \cite[Theorem 7.4.6]{bib:SC} we know that, under the assumptions {\bf (L1)}--{\bf (L4)}, for any $\eta \in \PP_{m_{0}}(\Gamma_{T})$ and any $x \in \R^{d}$, problem \eqref{min} is equivalent to the following one
\begin{equation*}
\inf_{u \in L^{\infty}(0,T;\R^{k})} J_{\eta}(x,u).
\end{equation*}
Thus, we can minimize over the set of bounded controls.
Let the control $u^{*} \in L^{\infty}$ be optimal for $V(t,x)$. By \eqref{DPPmfg} we deduce that for any $\epsilon\geq 0$
\begin{align*}
V(t,x)+\epsilon \geq \int_{t}^{t+\delta}{L(\gamma(s;t,x,u^{*}), u^{*}(s), m^{\eta}_{s})\ ds}+V(t+\delta, \gamma(t+\delta;t,x,u^{*})).
\end{align*}
Hence, we have that 
\begin{align}\label{ref}
\begin{split}
& V(t+\delta, x) -V(t,x) \leq V(t+\delta, x)-V(t+\delta, \gamma(t+\delta; t,x,u^{*}))
\\
- & \int_{t}^{t+\delta}{L(\gamma(s;t,x,u^{*}), u^{*}(s), m^{\eta}_{s})\ ds}+\epsilon
\\
\leq & \left(c_{2}Te^{T\| A \|} + c_{2}\sqrt{T}K+ \| G \|_{\infty}e^{T\|A\|} \right) |x-\gamma(t+\delta; t,x,u^{*})|+\delta\left( c_{1}+  \frac{1}{c_{0}}\| u^{*}\|_{\infty}\right),
\end{split}
\end{align}
where the last inequality holds true by the first step of the proof and assumption {\bf (L3)}. Moreover, since the curve $\gamma(\cdot; t,x,u^{*})$ is Lipschitz continuous in time, we know that the first term of the right-hand side is bounded by a constant times $\delta$. Thus, the proof of first estimate is complete.

On the other hand, again by \eqref{DPPmfg} we know that taking $u\equiv 0$ we have that
\begin{equation*}
V(t,x) \leq V(t+\delta, \gamma(t+\delta;t,x,0))+\int_{t}^{t+\delta}{L(\gamma(s;t,x,0), 0, m^{\eta}_{s})\ ds}. 
\end{equation*}
Therefore, adding and subtracting the term $V(t+\delta, x)$ we get that
\begin{align*}
V(t,x) -V(t+\delta,x) 
\\
\leq & V(t+\delta, \gamma(t+\delta;t,x,0))-V(t+\delta,x)+\int_{t}^{t+\delta}{L(\gamma(s;t,x,0), 0, m^{\eta}_{s})\ ds}.
\end{align*}
Hence, by the same considerations as in \eqref{ref} we get the result. 
\qed


\medskip
\medskip

\noindent {\bf Acknowledgements:} The authors would like to thank Pierre Cardaliaguet for fruitful discussions and comments on a preliminary version of this paper. The authors would like to express their gratitude to the anonymous reviewers for their careful reading of our manuscript and many insightful comments and suggestions. Piermarco Cannarsa was partly supported by Istituto Nazionate di Alta Matematica (GNAMPA 2019 Research Projects) and by the MIUR Excellence Department Project awarded to the Department of Mathematics, University of Rome Tor Vergata, CUP E83C18000100006. Cristian Mendico was partly supported by Istituto Nazionale di Alta Matematica (GNAMPA 2019 Research Projects). Part of this paper was completed while the second author was visiting the Department of Mathematics of the University of Rome Tor Vergata.

\end{document}